\documentclass[11pt,leqno]{amsart}
\usepackage{amsthm,amsfonts,amssymb,amsmath,oldgerm}
\numberwithin{equation}{section}
\usepackage{graphicx}
\usepackage{fullpage}

\def\eps{\varepsilon }

\newcommand\R{\mathbb R}
\newcommand\C{\mathbb C}

\def\eps{\varepsilon}

\newcommand\br{\begin{remark}}
\newcommand\er{\end{remark}}
\newcommand\bp{\begin{pmatrix}}
\newcommand\ep{\end{pmatrix}}
\newcommand{\be}{\begin{equation}}
\newcommand{\ee}{\end{equation}}
\newcommand\ba{\begin{equation}\begin{aligned}}
\newcommand\ea{\end{aligned}\end{equation}}

\newcommand{\bap}{\begin{app}}
\newcommand{\eap}{\end{app}}
\newcommand{\begs}{\begin{exams}}
\newcommand{\eegs}{\end{exams}}
\newcommand{\beg}{\begin{example}}
\newcommand{\eeg}{\end{exaplem}}
\newcommand{\bpr}{\begin{proposition}}
\newcommand{\epr}{\end{proposition}}
\newcommand{\bt}{\begin{theorem}}
\newcommand{\et}{\end{theorem}}
\newcommand{\bc}{\begin{corollary}}
\newcommand{\ec}{\end{corollary}}
\newcommand{\bl}{\begin{lemma}}
\newcommand{\el}{\end{lemma}}
\newcommand{\bd}{\begin{definition}}
\newcommand{\ed}{\end{definition}}
\newcommand{\brs}{\begin{remarks}}
\newcommand{\ers}{\end{remarks}}

\newcommand{\D }{\mathcal{D}}

\newcommand{\RR}{{\mathbb R}}

\newcommand{\TT}{{\mathbb T}}

\newcommand{\const}{\text{\rm constant}}
\newcommand{\Id}{{\rm Id }}

\newcommand{\blockdiag}{{\rm blockdiag }}

\newtheorem{theorem}{Theorem}[section]
\newtheorem{proposition}[theorem]{Proposition}
\newtheorem{corollary}[theorem]{Corollary}
\newtheorem{lemma}[theorem]{Lemma}

\theoremstyle{remark}
\newtheorem{remark}[theorem]{Remark}
\theoremstyle{definition}
\newtheorem{definition}[theorem]{Definition}

\newtheorem{example}[theorem]{Example}

\newcommand\cH{{\mathcal  H}}
\newcommand\cV{{\mathcal  V}}
\newcommand\cW{{\mathcal  W}}

\newcommand\cK{{\mathcal  K}}
\newcommand\cL{{\mathcal  L}}

\newcommand\cE{{\mathcal  E}}
\newcommand\cF{{\mathcal  F}}
\newcommand\cP{{\mathcal  P}}

\newcommand\cX{{\mathcal X}}

\newcommand\cM{{\mathcal M}}
\newcommand\cT{{\mathcal T}}
\newcommand\cS{{\mathcal S}}

\newcommand{\beq}{\begin{equation}}
\newcommand{\eeq}{\end{equation}}
\newcommand{\vp}{\varphi}

\font\tenronde=rsfs10
\font\sevenronde=rsfs7
\font\fiveronde=rsfs5
\newfam\rondefam
\scriptscriptfont\rondefam=\fiveronde
\textfont\rondefam=\tenronde
\scriptfont\rondefam=\sevenronde

\newcommand{\mez}{\frac{1}{2}}

\newcommand{\bN}{\mathbb{N}}
\newcommand{\bR}{\mathbb{R}}

\newtheorem{theo}{Theorem }[section]
\newtheorem{prop}[theo]{Proposition}
\newtheorem{cor}[theo]{Corollary}
\newtheorem{lem}[theo]{Lemma}

\newtheorem{ass}[theo]{Assumption}

\newtheorem{rem}[theo]{Remark}

\newtheorem{exams}[theo]{Examples}

\numberwithin{equation}{section}
\newtheorem{thm}{Theorem}[section]
\newtheorem{defn}[thm]{Definition}
\newcommand{\hu}{\hat{u}}

\title{Large-amplitude modulation of periodic traveling waves}

\author{G. M\'etivier}
\address{Universit\'e Bordeaux 1, Talence, France}
\email{ Guy.Metivier@math.u-bordeaux1.fr}
\author{Kevin Zumbrun}
\address{Indiana University, Bloomington, IN 47405}
\email{kzumbrun@indiana.edu}
\thanks{Research of K.Z. was partially supported
under NSF grant no. DMS-0300487}
\begin{document}

\begin{abstract}
We introduce a new approach to the study of modulation of high-frequency periodic wave patterns,
based on pseudodifferential analysis, multi-scale expansion, 
and Kreiss symmetrizer estimates like those in hyperbolic and hyperbolic-parabolic boundary-value theory.
Key ingredients are local Floquet transformation as a preconditioner removing large derivatives
in the normal direction of background rapidly oscillating fronts and 
the use of the periodic Evans function of Gardner to connect spectral information
on component periodic waves to block structure of the resulting approximately constant-coefficient resolvent ODEs.
Our main result is bounded-time existence and validity to all orders of large-amplitude smooth 
modulations of planar periodic solutions of multi-D reaction diffusion systems in the 
high-frequency/small wavelength limit.
\end{abstract}

\date{\today}
\maketitle


\tableofcontents


  \section{Introduction}\label{s:intro}

  As described in a variety of settings \cite{W,HK,DSSS,Se,NR}, there is a fascinating 
  connection between modulation of periodic traveling-wave solutions and quasilinear hyperbolic systems.
  This has led to rich mathematical interactions in both directions between the theories of pattern formation
  and hyperbolic and hyperbolic-parabolic systems \cite{SS,BNSZ,OZ1,OZ2,JZ,JNRZ1,JNRZ2,SSSU,JNRZ3}.
  In this paper, adapting ideas used in \cite{MZ,GMWZ1} to study 
  large-amplitude viscous shock and boundary layers of hyperbolic-parabolic systems, 
  we propose a new approach to the study of modulation in the high-frequency/zero-wavelength limit,
  designed for the treatment of {\it large-amplitude, multi-D} solutions.
Based on pseudodifferential analysis, multi-scale expansion, 
and Kreiss symmetrizer estimates originating in hyperbolic boundary-value theory \cite{K,Ma,M},
this yields in particular bounded-time existence and expansion to all orders of large-amplitude 
smooth modulations for reaction diffusion systems in the high-frequency limit.

\medskip
{\bf Derivative notation:} We use $\D_x$, $\D_t$, etc. to denote the usual differentials with respect to 
$x$, $t$ or other variables.
In the case of a single derivative $\D_x$ or $\D_{t,x}$, this is considered as a row vector
$(\D_{x_1}, \dots \D_{x_d})$ or $(\D_t, \D_{x_1}, \dots \D_{x_d})$, 
so that $\nabla_x=\D_x^T$, $\nabla_{t,x}=\D_{t,x}^T$ in vectorial notation.
Multiple derivatives $\D^j$, corresponding to multilinear maps, are used here only in describing various
error estimates, and for that purpose could equally well be thought of as the total derivative, or
set of all $j$th order derivatives.
For computations we will for ease of reading/writing
use, rather, summation notation with individual partial derivatives $\D_{x_j}$, $\D_t$
or vector calculus notation with the symbols $\nabla_x$, $\nabla_{t,x}$.

\subsection{Basic high-frequency modulation}\label{s:hfmod}

Consider a general reaction diffusion system
 \begin{equation} \label{eq1}
 \eps  \D_t u + f(u)  = \eps^2 \Delta_x u , \quad x\in \R^d, \, u, f \in \R^n,
\end{equation}
in the high-frequency/small-wavelength limit $\eps\to 0^+$.
This corresponds, under the rescaling $(t,x)\to (t/\eps, x/\eps)$ to solutions $u(t/\eps, x/\eps)$
of the unscaled system $\D_t u+ f(u)= \Delta_x u$. 

\begin{ass} \label{ass1}
 There is  a (nontrivial) smooth function $p(k, \theta) $ of $k \in \cK \subset \RR^d $, $\cK$ bounded and
bounded away from the origin, and $\theta$, 
 $2\pi $- periodic in $\theta$, and a smooth function $\omega (k)$ such that  
 \begin{equation}
 \label{prof0}
\omega (k)   \D_\theta p  + f( p ) =  | k | ^2 \D_\theta^2 p.
 \end{equation}
 \end{ass}

 Assumption \ref{ass1} implies that for all $k\in \cK$, 
\begin{equation} \label{planar}
 \overline u (t,x) = p (k, \psi _k (t, x) / \eps), \qquad \psi_k (t,x) = k \cdot x + \omega (k) t 
 \end{equation} 
 is an exact solution of \eqref{eq1}. 
Consider now a {\it smooth modulation} 
\begin{equation}\label{mod}
u^{\eps, a}   (t, x)   =  p(k , \psi/ \eps), 
\end{equation}
where $k$ and $\psi$ are smooth enough functions of $(t, x)$ and $p$ is a smooth enough function of $k$
and $\theta$, with truncation error 
 \begin{equation}
 \label{err}
	 R^{\eps} :=  \eps  \D_t u^{\eps,a} + f(u^{\eps,a})  - \eps^2 \Delta u^{\eps,a}.   
\end{equation}

Computing
	 \ba\label{remcomp}
  \eps  \D_t  u^{\eps, a} &=  \D_t \psi  \D_\theta p + \eps \sum_l  \D_t k_l \D_{k_l}  p ,\\
  \eps  \D_{x_j}  u^{\eps, a} &=  \D_{x_j} \psi  \D_\theta p +\sum_l   \eps \D_{x_j} k_l  \D_{k_l}  p, \\
  \eps^2  \Delta_x   u^{\eps, a}   &=    | \D_x \psi |^2  \D^2_\theta p + \eps \Delta \psi  \D_\theta p 
  + 2  \eps \sum_{j, l}  \D_{x_j} \psi   \D_{x_j} k_l  \D^2_{\theta, k_l} p 
  \\
  \quad & + \eps ^2 \sum_{j, l, m}  \D_{x_j} k_l  \D_{x_j} k_m \D^2_{  k_l, k_m} p , 
  \ea
  and combining, we have that the main ($O(\eps^0)$) term in $R^{\eps} $ is 
 $$
 \D_t \psi  \D_\theta p+ f(p) -   | \D_x \psi |^2  \D^2_\theta p = 
 ( \D_t \psi - \omega (k) )  \D_\theta p+  ( | k|^2 - | \D_x \psi |^2 )  \D^2_\theta p .
 $$
 This vanishes when 
 \begin{equation}\label{modcon}
 \D_t \psi = \omega (\nabla_x \psi) , \qquad  k = \nabla_x \psi,
 \end{equation}
 yielding truncation error $R^\eps= O(\eps)$ formally validating the high-frequency approximation \eqref{mod}.

\subsection{Higher-order expansion}\label{s:allorders}
Continuing, we seek a general {\it multi-scale expansion} \cite{JMR,GMWZ1}
 \begin{equation}\label{highexp}
	 u^{\eps,m} (t,x) = U_{\eps,m} (t, x, \frac{1}{\eps}\psi(t,x) +  \vp_{\eps,m-1}(t, x)   ) 
 \end{equation} 
 (with convention $\vp_{\eps, -1}=0$) with 
 \begin{equation}\label{highdef}
	 U_{\eps,m}(t,x, \theta)  = \sum_{ n=0}^m \eps^n U_n (t, x, \theta) , 
 \qquad
	 \vp_{\eps,m} (t,x) = \sum_{n = 0}^{m} \eps^n \vp_n(t,x) , 
 \end{equation} 
 satisfying the consistency condition
 \be\label{acon}
 R^{\eps,m}:= \eps  \D_t   u^{\eps, m} +  f ( u^{\eps, m})   -   \eps^2 \Delta_x  u^{\eps, m}
 = O(\eps^{m+1})
 \ee
on the residual, or truncation error $R^{\eps,m}$ of the approximate solution $u^{\eps,m}$. 

\begin{ass}\label{ass2}   
For $k \in \cK$, consider the linearization
	\begin{equation}\label{L}
L(k, \D_\theta) = - \omega(k) \D_\theta -  f'(p (k, \theta))  + | k|^2 \D_\theta^2, 
\end{equation} 
of profile equation \eqref{prof0} about $p$, satisfying (by translation invariance)
\begin{equation}
L(k, \D_\theta) \D_\theta p  = 0. 
\end{equation}
Assume that $0$ is a simple eigenvalue of $L(k, \D_\theta)$ in $L^2 (\TT)$, 
with associated eigenfunction $\D_\theta p(k, \, \cdot \, )$.
\end{ass}

Assumption \ref{ass2} is equivalent to {\it transversality} of solutions $p$ of the profile ODE
\eqref{prof0}, which may be recognized as a slightly strengthened, {linear version} of
Assumption \ref{ass1}.

 \br\label{kzerormk}
 For $f$ sufficiently smooth, and any $k\neq 0$ and transversal solution $p$, $\omega$ of \eqref{prof0},
 Assumptions \ref{ass1}-\ref{ass2} hold locally near $(k,\omega,p)$ 
 by the Implicit Function Theorem and smooth dependence of solutions of ODE.
 The limit $k\to 0$, explicitly excluded here, is singular for \eqref{prof0}, hence smooth dependence would
 not necessarily hold at $k=0$.
 More important, $0<|k|<\infty$ excludes spatial periods $X=2\pi \eps /|k|$ of $\infty$ or $0$, as
 will be important at a technical level later on.
 We note that the condition $\D_x \psi=k\neq 0$ is quite restrictive on the possible geometry of
 wave fronts of $u^{\eps,a}$, given by level sets of $\psi$, implying foliation of $\R^d$.
 In particular, there can be no closed level surfaces. 
 \er

Let $H^s$ denote the usual Sobolev spaces with respect to $x$ and $t$, and 
$\cH^s_\eps$ the spaces defined by norms
$\|h\|_{\cH^s_\eps}^2:= \sum_{j=1}^s \eps^{2j} \|\D_{t,x}^j h\|_{L^2}^2$
accomodating presence of the fast variable $\theta=\vp/\eps$.
Note that $\eps^{-1/2}\|\cdot\|_{\cH^s_\eps}$ is equivalent to the usual Sobolev norm in
rescaled variables $(x',t')=(x/\eps, t/\eps)$, hence we have 
for $s\geq [d/2] + 1$, where $[\cdot]$ denotes least integer,
the Sobolev embedding
\be\label{sob}
\|h\|_{L^\infty}\leq C\eps^{-1/2}\|h\|_{\cH^s_\eps}.
\ee
Here and elsewhere, all norms are with respect to $\R^d\times [0,T]$, unless otherwise specified,
in which case we shall write $H^s[T_1,T_2]$, $\cH^s_\eps[T_1,T_2]$, etc. to denote norms with
respect to $\R^s\times [T_1,T_2]$.

With these preparations, we have the following result, established in Section \ref{s:asymptotic}
by induction building on computations \eqref{err}-\eqref{modcon}.

\begin{theo}\label{main1}
Under Assumption \ref{ass2}, if $\psi$ is a  $C^\infty$ solution of eikonal equation 
$\D_t \psi = \omega (\nabla_x \psi)$ on $[0, T] \times \RR^d$, then there are $C^\infty$ 
asymptotic solutions \eqref{highexp}, \eqref{acon} at all orders, with 
\begin{equation}
\label{symb0}
U_0 = p(k(t,x), \theta) , \qquad k = \nabla_x \psi.
\end{equation}
More precisely, for $f\in C^{2(m+1)}$, $\psi\in C^{2(m+1)}$, there are $L^\infty$ solutions to order $n\leq m$, with 
\be\label{asinftybds}
	\|(U_n, \vp_{n-1})\|_{L^\infty}\lesssim\|\D_{t,x}\psi\|_{C^{2n}},
	\quad
	\|R^{\eps,m}\|_{L^\infty} \lesssim \eps^{m+1}\|\D_{t,x}^2\psi\|_{C^{2(m+1)}}
	\ee
	and for $f\in C^{s+1+2(m+1)}$, $\D^2_{t,x} \psi\in H^{s+ 2(m+1)}$, 
	$s\geq 0$, there are $H^s$ solutions to order $n\leq m$, satisfying
\be\label{asHbds}
	\|(U_n, \vp_{n-1})\|_{\cH^s_\eps}\lesssim\|\D_{t,x}\psi\|_{H^{s+2n}},
	\quad
	\|R^{\eps, m}\|_{\cH^s_\eps} \lesssim \eps^{m+1}\|\D_{t,x}^2\psi\|_{H^{s+ 2(m+1)}},
	\ee
	in each case satisfying \eqref{symb0}.
\end{theo}

\br\label{geormk}
In what follows, we shall assume $\D_{t,x}\psi\in C^{2m}$, $\D_x^2\psi\in H^{s+ 2(m+1)}$, 
to obtain
	for $n\leq m$
\be\label{usedbds}
	\|(U_n, \vp_{n-1})\|_{L^\infty}\lesssim\|\D_{t,x}\psi\|_{C^{2m}}, \quad 
	\|R^{\eps,m}\|_{\cH^s_\eps} \lesssim \eps^{m+1}\|\D_{t,x}^2\psi\|_{H^{s+ 2(m+1)}}.
	\ee
By $\D_t \psi=\omega(\nabla_x \psi)\lesssim \D_x \psi$ and $k=\nabla_x \psi$, these assumptions may be equivalently
phrased as
\be\label{kbds}
k\in C^{2m}, \quad \D_x k \in H^{s+ 2(m+1)},
\ee
implying in particular $k\in L^\infty$ and $\D_x k\in L^2$. The first condition corresponds to
$k\in \cK$ in Assumption \ref{ass1}.
The second imposes additional geometric constraints; for example, radially symmetric configurations of
$k$ are disallowed in dimensions $d\geq 2$, as (since by assumption they are bounded from zero) 
their gradients decay no faster than $|x|^{-1}$, hence just miss being bounded in $L^2$.
Moreover, in all dimensions, $\D_x^2 \psi \in H^{[d/2]+1}$ implies by Sobolev embedding
\be\label{zerolim}
\hbox{\rm $\D_x^2 \psi \to 0$ as $|x|\to \infty$,}
\ee
giving $\D_x k \to 0$ as $|x|\to \infty$.
\er


\subsection{Linear estimates and nonlinear validation}\label{s:linearized}
Next, given an approximate solution $u^{\eps,m}$, we estimate the remainder, or convergence error,
\be\label{remainder}
v^{\eps,m} := u- u^{\eps, m}, 
\ee
where $u$ is an exact solution.  
The equation for $ v = u- u^{\eps, a}$ may be expressed as
 \begin{equation}
  \label{eq2intro}
	 \eps \cP_{u^{\eps,m}}v^{\eps,m} =  -  R^{\eps,m}  + Q(u^{\eps,m}, v^{\eps,m}) := \eps e^{\eps,m} ,
 \end{equation} 
 where
 \be\label{cPdef}
\eps \cP_{u^{\eps,m}}:= \eps  \D_t   +  g^\eps     -   \eps^2 \Delta_x,
\qquad g^\eps (t, x) =  f' (u^{\eps, m}) = G (k, \Psi/ \eps)
\ee
is the linearization of \eqref{eq1} about $u^{\eps,m}$ and
$$
Q:= -\big(f(u^{\eps, m}+v)- f(u^{\eps, m}) - f'(u^{\eps, m})v \big),
$$
the Taylor remainder, is quadratic in $v$ for $\|v\|_{L^\infty}\leq C$.

\begin{ass}\label{ass3}   
For $k \in \cK$, consider the Bloch--Fourier operator
	\begin{equation}\label{B-F}
L_{\xi,\eta}(k, \D_\theta) = -\omega(k) (\D_\theta+i\xi) -  f'(p (k, \theta))  + | k|^2 (\D_\theta+i\xi)^2 - |\eta|^2
\end{equation} 
$\xi\in \R$, $\eta=(\eta_2, \dots, \eta_d)\in \R^{d-1}$ associated with linearized operator \eqref{L},
where $\xi$ is Floquet number and $\eta_2, \dots, \eta_d$ are Fourier frequencies in directions transverse to $k$.
Denoting by $\sigma(L)$ the spectrum of a linear operator $L$, assume that
(i) $L_{0,0}(k, \D_\theta) $ has a simple eigenvalue at $\lambda=0$ and no other pure imaginary eigenvalues, and
(ii) $\Re \sigma \big( L_{\xi,\eta}(k, \D_\theta) \big)\leq - c |(\xi,\eta)|^2$ for some $c>0$.
\end{ass}

Assumption \ref{ass3} may be recognized as the {\it diffusive stability} condition of Schneider \cite{S1},
sufficient \cite{S1,JZ,JNRZ1,JNRZ2,SSSU,JNRZ3} for linearized and nonlinear stability of component planar 
periodic waves \eqref{planar}.
The center of our analysis, and the main contribution of this paper,
is the following result converting the ``local'' spectral stability condition of 
Assumption \ref{ass3} to a global linear estimate.

\begin{theo}\label{main2}
Under Assumption \ref{ass3}, together with 
	$f\in C^{s+1+2(m+1)}$, $\D_x^2\psi\in H^{s+ 2(m+1)}[-T_0,T]$, $\D_{t,x}\psi\in C^{2m}[-T_0,T]$,
	$T_0,T>0$, for every $h\in \cH^s_\eps[-T_0,T]$ with $h$ vanishing for $t<0$, the problem
\be\label{linprob}
\cP_{u^{\eps,m}}v =  h
\ee
	has a unique solution $v\in \cH^{s}_\eps[0,T]$ vanishing for $t<0$, satisfying
\be\label{linbd}
\|v \|_{\cH^{s+1}_\eps} \lesssim \|h \|_{\cH^{s}_\eps}.
\ee
\end{theo}

\br\label{comprmk}
Problem \eqref{linprob} may be recognized as an exact analog of \cite[Problem (1.25)]{MZ}, and
Theorem \ref{main2} as a simplified (since not involving conormal derivatives or $L^\infty$ norms
required in the boundary-layer case treated there) version of \cite[Thms. 1.9-1.10]{MZ}.
Similarly as in the boundary-layer case,
applying detailed estimates of \cite{JZ}, appropriately rescaled in $\eps$, one may verify that
linear bound \eqref{linbd} is sharp in the case $k\equiv \const$ of an exactly periodic planar wave.
\er

\br\label{multirmk}
Evidently, for the simple isotropic (Laplacian) diffusion considered here, multi-D diffusive spectral stability 
is equivalent to 1-D diffusive stability, $\eta\equiv 0$, with $\sigma (L_{\xi,\eta})= \sigma (L_{\xi,0} -|\eta|^2).$
\er

Theorem \ref{main2} is established, following the general strategy of \cite{MZ}, by 1. first reducing via 
local coordinate change/pseudodifferential calculus effectively to the corresponding problem on the pseudodifferential
symbol for the local planar problem, then 
2. removing fast oscillations by a periodic Floquet transformation
in the normal direction conjugating the problem to approximately constant-coefficient, and finally 
3. obtaining linearized estimates on the resulting ``averaged'' system by frequency-dependent ``Kreiss-type'' 
energy estimates as in \cite{K,Ma,MZ,GMWZ1}.

Whereas the key conjugation in step 2 was carried out in \cite{MZ} by an asymptotically constant-coefficient
coordinate change (guaranteed by the ``conjugation lemma'' of \cite{MZ}), 
reducing to a known limiting constant-coefficient problem,
the Floquet transformation used here results in a constant-coefficient averaged system that is a priori unknown.
We make important use in step 3 of the {\it periodic Evans function} of Gardner \cite{G} in deducing 
needed averaged structure
from spectral information encoded in Assumption \ref{ass3}.
These steps are carried out in Section \ref{s:remainders}.

From \eqref{linbd}, we readily obtain the following nonlinear convergence result validating 
\eqref{mod}, \eqref{highexp}.

\begin{cor}\label{main3}
Under Assumption \ref{ass3},
	for $f\in C^{s+1+2(m+1)}$, $\D_x^2\psi\in H^{s+ 2(m+1)}$, $\D_{t,x}\psi\in C^{2m}$,
$m\geq 2$, $s>[d/2]$, there is for $\eps>0$ sufficiently small
	a unique solution $u\in \cH^{s+1}_\eps$ of \eqref{eq1}
	satisfying	
	\be\label{presdata}
u|_{t=0}=u^{\eps,m}|_{t=0}
\ee
and
\be\label{rembd}
\|u-u^{\eps,m}\|_{\cH^{s+1}_\eps} + \eps^{1/2} \|u-u^{\eps,m}\|_{C^{s -[d/2]}_\eps} 
\lesssim \eps^{m}\|\D_{t,x}^2\psi\|_{H^{s+ 2(m+1)}}.
\ee
\end{cor}

\br\label{correctrmk}
Similarly as in \cite{MZ}, we need additional correctors (two here as compared to one in \cite{MZ})
to the basic high-frequency modulation $u^{\eps,a}=u^{\eps,0}$ in order to close the analysis.
\er

Corollary \ref{main3} yields bounded-time existence and rigorous order-$m$ expansion of (possibly) large-amplitude 
modulations $\psi$ satisfying $|\D_{t,x} \psi|_{C^0} , \, \|\D_{t,x}^2\psi\|_{H^{s+ 2(m+2)}}<\infty $
for $m\geq 2$, $s>[d/2]$.

By expanding to an extra order, and combining \eqref{rembd} with \eqref{asHbds}, we find 
for $f\in C^{s+1+2(m+2)}$, $\D_x^2 \psi\in H^{s+ 2(m+2)}$ the improved, $O(\eps^{m+1})$ remainder estimate
\be\label{optrembd}
\|u-u^{\eps,m}\|_{\cH^{s+1}_\eps}\leq 
\|u-u^{\eps,m+1}\|_{\cH^{s+1}_\eps} + \|u^{\eps,m+1}-u^{\eps,m}\|_{\cH^{s+1}_\eps} 
\lesssim \eps^{m+1}\|\D_{t,x}^2\psi\|_{H^{s+ 2(m+2)}}.
\ee
For $m>s$, this implies also the useful standard Sobolev embedding estimate
$$
\|u-u^{\eps,m}\|_{L^\infty} \lesssim \eps^{ m+ 1/2 }\|\D_{t,x}^2\psi\|_{H^{s+ 2(m+2)}}.
$$

\subsection{Discussion and open problems}\label{s:disc}
Corollary \ref{main3} may be recognized, after accounting different scalings, as a multidimensional analog of
the 1D results \cite[Thms. 6.1--6.2]{DSSS}, obtained by quite different techniques.
As noted in \cite{DSSS}, solutions of \eqref{modcon} in general become singular in finite time, hence the bounded-time
assumption is natural for smooth solutions.
However, the analysis of \cite{DSSS}, proceeding by normal form reduction, approximates the function
$\omega(k)$ appearing in modulation equation \eqref{modcon} by its second-order Taylor expansion about a reference
state $k=k_*$, hence is inherently limited to the small-amplitude case $|k-k_*|\ll 1$.
By contrast, our restrictions on $k$ in Assumption \ref{ass1}
are only to guarantee certain 
natural properties of the associated periodic traveling-wave solutions, 
allowing the treatment of large-amplitude solutions $|k-k_*|\gg 1$.

An advantage of our approach is that we obtain a ``prescribed data'' result \eqref{presdata}
yielding a unique exact solution $u$ satisfying the initial data of approximate solution $u^{\eps,m}$,
whereas the corresponding results stated in \cite{DSSS} are ``prepared data'' type, asserting
existence of a nearby exact solution for {\it some}, unspecified, initial data.
More (see Remark \ref{attractrmk}), our analysis gives the ``approximate attraction'' property that
exact solutions of \eqref{eq1} with $u|_{t=0}$ $\eps^{m+1}$-close to $u^{\eps,m}$ in $\cH_\eps^s(\R^d)$
(i) exist up to the full interval of existence $[0,T]$ of $u^{\eps,m}$, and (ii) remain $\eps^m$-close
in $\cH^s_\eps[0,T]$, thus justifying the idea (see \cite{CE,S2,vH,JNRZ3} in related contexts)
of modulation expansions as approximate attracting manifolds.
We discuss these issues further in Section \ref{s:nonlin}.

On the other hand, our use of standard Sobolev norms imposes $\|\D^2\psi\|_{L^2}<\infty$, whereas
the analysis of \cite[Thms. 6.1--6.2]{DSSS}, based on locally square-integrable norms
requires $\|\D_x \psi \|_{L^\infty}\ll 1$, imposing smallness but no localization.
It would be interesting to try to extend our results to the case
$\|\D_x \psi \|_{L^\infty}< \infty$, perhaps by the use of local Sobolev norms as in \cite{DSSS}.
Another interesting direction suggested by the results of \cite{DSSS} would be to incorporate
the diffusive scaling of \cite[\S 4]{DSSS}, leading in place of modulation equation \eqref{modcon}
to a second-order diffusive regularization, and allowing the treatment of additional interesting solutions
such as viscous shock profiles.

The restriction in \eqref{eq1} to isotropic, Laplacian diffusion was made for convenience/simplicity in exposition,
and should in principle be extensible to general strictly parabolic diffusions.  However, in the general case,
the multi-D stability analysis does not reduce to 1-D. A very interesting open problem is to construct Kreiss
type symmetrizers in this ``truly multi-D'' case.
It is interesting to note that the phenomenon of ``glancing'' treated in Appendix \ref{s:kcon}
arises here already in the 1-D case, whereas in hyperbolic BVP theory it is associated only with multi-D
phenomena; this agrees with the intuition that 1-D periodic theory is roughly 1.5-D, due to the incursion of an
additional Floquet number $\xi$ along with the spatial variable $x$.

As a final open problem, we mention the extension of our results to systems of conservation laws or relaxation systems,
for which the associated formal modulation system analogous to \eqref{modcon} is no longer scalar, but of system form.
See, for example the 1-D analysis of \cite{NR}.

  \medskip
  {\bf Acknowledgement.} 
  The second author thanks University of Bordeaux I for its support and hospitality during a March 2015 visit
  in which this project was initiated and partly carried out,
  and Benjamin Melinand for helpful conversation regarding prescribed data and
  initial time layers. Thanks also to the anonymous referees for their careful reading and helpful suggestions
  improving the exposition:
  in particular pointing out the interesting connection to \cite{C} noted in Remark \ref{cormk}
  and suggesting a related improvement (the addition of term $|\tau + \sum_j c_j  \eta_j|^2$ on the right-hand side
  of bound \eqref{symmeq}) in Proposition \ref{symm}.

 
 \section{Asymptotic solutions}\label{s:asymptotic}
 We begin by deriving \eqref{highexp}-\eqref{acon}, starting from first principles, 
 rederiving \eqref{mod} in the course of the analysis.
 We look for asymptotic solutions 
 \begin{equation}
 u^\eps (t,x) = U_\eps (t, x, \frac{1}{\eps}  \psi(t, x) + \vp_\eps(t,x)  ) 
 \end{equation} 
 with 
 $U_\eps(t,x, \theta)  \sim \sum_{n \ge 0} \eps^n U_n (t, x, \theta)$, $2\pi$-periodic in $\theta$, and 
 $\vp_\eps (t,x) \sim \sum_{n \ge 0} \eps^n \vp_n(t,x)$ . 
We have
 \begin{equation*}
 \eps \D_t u_\eps \sim \D_t \psi \D_\theta U_0 + \sum_{n \ge 0}
 \eps ^{n +1} \Big( \D_t \psi D_\theta U_{n+1} + \D_t U_n  + 
 \sum_{p= 0}^n\D_t \vp_p \D_\theta U_{n-p} \Big)  
 \end{equation*}
  \begin{equation*}
  \begin{aligned}
  \eps^2  \Delta_x   u^{\eps}   \sim   
   | \D_x \psi |^2  \D^2_\theta U_0  + \eps  \Big( | \D_x \psi |^2  \D^2_\theta U_1 +  \Delta \psi  \D_\theta U_0 
 &  +    \eps \sum_{j, l}  \D_{x_j} \psi   \D_{x_j} k_l  \D^2_{\theta, k_l} U_0 \Big) 
  \\
 +  \sum_{n \ge 0} \eps^{n+2} \Big( | \D_x \psi |^2  \D^2_\theta U_{n+2}  +  \Delta \psi  \D_\theta U_{n+1} 
  +  & 2   \sum_{j, l}  \D_{x_j} \psi   \D_{x_j} k_l  \D^2_{\theta, k_l} U_{n+1}  
  \\
  &  +   \sum_{j, l, m}  \D_{x_j} k_l  \D_{x_j} k_m \D^2_{  k_l, k_m} U_n \Big)  , 
 \end{aligned} 
 \end{equation*}
 Thus,  plugging the expansion into the equation we get 
   \begin{equation*}
 \label{eq1asym}
 \eps  \D_t u^\eps + f(u^\eps)  - \eps^2 \Delta_x u^\eps \sim   \sum_{n \ge  0 } \eps^n F_n 
\end{equation*}
with  
\begin{equation*}
F_0 = \D_t \psi \D_\theta U_0 + f(U_0) - | \D_x \psi |^2 \D_\theta^2 U_0 
\end{equation*}
and 
  \begin{equation*}
F_{n+1}  =  \cL (t, x , \theta,  \D_\theta)  U_{n+1}   + \D_t \vp_n \D_\theta U_0  + \cF_{n},
\end{equation*}
where
 \begin{equation*}
\begin{aligned}
\cL   (t, x, \theta,  \D_\theta) &= 
	\D_t \psi \D_\theta  +    f'(U_0(t,x, \theta) ) - | \D_x \psi |^2 \D_\theta^2   
	= L(k(t,x), \D_\theta )
	 \end{aligned}
\end{equation*}
and $\cF_n$ depends only on $(U_0, \ldots, U_n)$ and $(\vp_0, \ldots , \vp_{n-1})$ and their derivatives. 
  Thus, we get an asymptotic solution if we solve by induction 
  \begin{equation}
 \label{eqas0}
	  L(k(t,x),\D_\theta)U_0=  0
\end{equation}
  and 
    \begin{equation} \label{eqasn}  
	     L (k(t, x ) ,  \D_\theta)  U_{n+1}   + \D_t \vp_n \D_\theta U_0  = -  \cF_{n}. 
\end{equation}

Under Assumption \ref{ass2}, $0$ is a simple eigenvalue in $L^2 (\TT)$ of
$$
L(k, \D_\theta) = \omega(k) \D_\theta +  f'(p (k, \theta))  - | k|^2 \D_\theta^2
$$
with associated eigenfunction $\D_\theta p(k, \, \cdot \, )$. 
In this case, the range of $L$ is of codimension one and there is $ h(k, \cdot)$, $2\pi$-periodic in 
$\theta$, in the kernel of the adjoint $L^*$,
thus smooth in $\theta$ and in $k$ because of the constant multiplicity, with 
\begin{equation}
\int h(k, \theta) \D_\theta p (k, \theta) d  \theta = 1 
\end{equation}
and such that $f$ belongs to the range of $L(k, \D_\theta)$ if and only if 
\begin{equation}
\label{range}
\int h(k, \theta) f( \theta) d  \theta = 0.  
\end{equation} 
Moreover, there is a partial inverse $R(k) $ of $L(k, \D_\theta)$ such that
$R(k) \D_\theta p(\cdot, k) = 0$ and  if $f$ satisfies 
\eqref{range}, then $L (k, \D_\theta) R(k) f = f $. 

\begin{lem}
\label{lemsmooth}
	For $f, p(k,\cdot)\in C^\infty$, the operator $R(k)$ maps $C^\infty (\TT)$ into $C^\infty (\TT)$ 
and depends smoothly on $k \in \cK$. 
	Moreover, for $f\in C^{m+1}$, $p(k,\cdot) \in C^{m}$, $m\geq 0$, it is bounded on $C^{m}(\TT)$.
\end{lem}

\begin{proof}
	The statement is clear for $m=0$. For higher $m$, it may be obtained by induction, applying $\D_\theta^m$
	to the equation $Lw=\tilde f$ for $w=Rf$, 
	where $\tilde f:=(\Id-\Pi_0)f$ is the projection of $f$ onto $h^\perp$,
	with 
	$$
	\|\tilde f\|_{C^m}\lesssim \|f\|_{C^m},
	$$
	then rearranging to express $ L (\D_\theta^m w)$ as $ \D_\theta^m \tilde f$ 
	plus the sum of products of lower $\theta$-derivative terms in $w$ and the variable coefficient
	$f'(p(k,\theta)$ in $L$, the latter bounded by our smoothness assumptions on $f$ and $p$. 
	This yields $\|\D_\theta^m Rf\|_{L^\infty}= \|\D_\theta^m w \|_{L^\infty}
	\lesssim \|\tilde f\|_{C^m} \lesssim \| f\|_{C^m}$ as claimed.
\end{proof}

\begin{proof}[Proof of Theorem \ref{main1}]
The first equation \eqref{eqas0} is satisfied if  $U_0$ is given by \eqref{symb0}. 
Next, \eqref{eqasn} is satisfied if
\begin{equation} 
\D_t \vp_n (t,x) = - \int h(k(t, x), \theta)  \, \cF_n (t, x, \theta)  d \theta 
\end{equation} 
and 
\begin{equation}
U_{n+1} (t, x, \theta) =  -   R(k(t, x))\big(  \cF_n - \D_t \vp_n (t,x) \D_\theta U_0) . 
\end{equation}
By Lemma~\ref{lemsmooth}, one checks by induction that for $f, \psi \in C^\infty$
the $U_n$ and $\vp_n$ are $C^\infty$ functions of $(t, x, \theta)$.
Moreover, for $f\in C^{s+2(m+1)}$, $\psi\in C^{2(m+1)}$ and $f\in C^{s+1+2(m+1)}$,
$\psi\in H^{s+ 2(m+1)}$, respectively, one has bounds \eqref{asinftybds} and \eqref{asHbds}, respectively.
\end{proof} 


 \section{Linear estimates}\label{s:remainders}
 
 Consider an approximate solution 
 \begin{equation}
 u^{\eps, m} (t,x) = \sum_{ n =0}^m  \eps^n U_n (t, x, \Psi / \eps ) , \ 
	 \quad \Psi = \psi +  \eps \sum_{n=0} ^{m-1} \eps^n \vp_n  = \psi + \eps \vp , 
 \end{equation}
 where for ease of writing we have suppressed dependence of $\Psi$ upon $\eps,t,x$.
 Then
 \begin{equation}
  \label{err1}
 \eps  \D_t   u^{\eps, m} +  f ( u^{\eps, m})   -   \eps^2 \Delta_x  u^{\eps, m} =   R^{\eps,m}   
	 = O (\eps^{m+1}) . 
 \end{equation} 
The  equation for $ v = u- u^{\eps, m}$ is 
 \begin{equation}
  \label{eq2}
 \eps  \D_t v  +  g^\eps  v   -   \eps^2 \Delta_x v  =  -  R^{\eps,m}  + Q(u^{\eps,m}, v) := \eps e^\eps . 
 \end{equation} 
 where 
 $g^\eps (t, x) =  f' (u^{\eps, m})$ 
 and $Q$ is quadratic in $v$ for bounded $v$.

 Local to any point $(t,x)=(\underline t, \underline x)$, we may introduce new spatial coordinates
 \begin{equation}
	 z = \Psi (t, x), \; y=(y_2, \dots, y_d) 
 \end{equation}
 with $\nabla_x{y_j}$ orthonormal and orthogonal to $\nabla_x z $ at $(\underline x, \underline t)$,
 and $y_j$ constant along integral curves of $\nabla \Psi$, hence $\D_x y \nabla_x z=
 \D_x y \nabla_x \Psi =0 $ at all $(t,x)$.
 Here, we are using the assumed property, inherited for $|\nabla_x \Psi|$, that $|\nabla_x  \psi|$ 
 is bounded and bounded from zero. 
 
 Then, 
 \ba\label{trans}
 \D_t v&= \D_z v \D_t z +\D_y v\D_t y +  D_t v=  D_z v \D_t \Psi +\D_y v\D_t y +  D_t v,\\  
 \D_x v&= \D_z v \D_x z + \D_y v \D_x y = \D_z v \D_x \Psi + \D_y v \D_x y ,\\
 \ea
 and so, using $\D_x y \nabla_x \Psi =0 $, and orthonormality of $\{\nabla_x y_j\}$ at 
 $(t,x)=(\underline t, \underline x)$, 
 $$
 \begin{aligned}
	 \Delta_x v&= \nabla_x \cdot (\D_x v)^T
	 = \nabla_x \cdot   \big(\nabla_x \Psi \D_z v\big) +\nabla_x \cdot \big(\sum_j \nabla_x {y_j}\D_{y_j} v\Big)\\
	 &=  \big((\Delta_x \Psi)\D_z v + \nabla_x (\D_z v)\cdot \nabla_x \Psi\big)
	 + 
	 \sum_j \big( 
	 \nabla_x(\D_{y_j} v)  \cdot \nabla_x {y_j} + \Delta_x {y_j} \D_{y_j} v \big)
	 \\
	 &=  \big((\Delta_x \Psi)\D_z v + \D_x (\D_z v) \nabla_x \Psi\big)
	 + 
	 \sum_j \big( \D_x (\D_{y_j} v)  \nabla_x {y_j} + \Delta_x {y_j} \D_{y_j} v \big)
	 \\
	 & = 
	  (\Delta_x \Psi) \D_z v + |\D_x \Psi|^2 \D_z^2 v + 2 (\D_y\D_z v) (\D_xy \nabla_x \Psi) ) 
	  +  \sum_{ij}(\D_{y_i}\D_{y_j} v)(\nabla_x y_i\cdot \nabla_x y_j) \\
	  &\quad +\sum_j (\Delta_x {y_j}) \D_{y_j} v \\
	  &=
	  |\D_x \Psi|^2 \D_z^2 v + \Delta_y v + (\Delta_x \Psi) \D_z v + o(1)|\D^2_y v| + O(|\D_{y_j} v|), 
 \end{aligned}
$$
where $o(1)\to 0$ as $(t,x)\to (\underline t, \underline x)$.

 Thus, in a neighborhood of $(\underline x, \underline t)$,  \eqref{eq2} is transformed to
\ba\label{eq3}
 \eps  \D_t v  + \eps \sum_j c_j \D_{y_j} v
 +&g^\eps  v   + \eps  a \D_z   v  - \eps^2  |k|^2  \D_z^2 v -\eps^2\Delta_y v\\
 &+ o(1)\eps^2|\D_y^2 v| + O(\eps^2)|\D_y v|=
	  - R^\eps +  Q(u^{\eps,m}, v) := \eps e^\eps   ,
 \ea
 where
 \begin{equation}\label{kaG}
	 k:=\D_x \Psi, \quad a = \omega (k) + \eps a_1(t,x), \quad
 c_j=\D_t y_j, \quad
	 \hbox{\rm and $g^\eps (t, x) =  G (k, z/ \eps) + \eps g_1(t,x,z/\eps)$,} 
 \end{equation}
 with 
 $$
 \begin{aligned}
	 a&= \eps^{-1}\D_t \Psi +  \Delta_x \psi - \eps^{-1} \omega(\nabla_x \Psi)\\
	 &=  \eps^{-1} \D_t \psi +  \D_t \phi +  \Delta_x \psi - \eps^{-1}\omega(\nabla_x \psi) +   
	 \eps^{-1}\big(\omega(\nabla_x \psi) - \omega(\nabla_x \Psi)\big)\\
	 &= \D_t \phi + \Delta_x \psi + 
	 \eps^{-1}\big(\omega(\nabla_x \psi) - \omega(\nabla_x \psi + \eps \nabla_x\phi)\big),\\
	 G(k,\cdot)&:= f'(U^0(t,x, \cdot))=f'(p(k, \cdot)),\\
	 g_1&= f'(u^{\eps,m}(t,x, \cdot))- f'(U^0(t,x, \cdot))= f'\big(U^{0}(t,x, \cdot)+\eps\sum_{j=1}^m 
 \eps^{j-1}U^{j}(t,x, \cdot)\big)- f'(U^0(t,x, \cdot)
 ,
 \end{aligned}
 $$
 and $o(1)\to 0$ with the size of the neighborhood about $(\underline x,\underline t)$.
 Here, $g_1$ is controlled in relevant norms by 
 $\sum_{j=1}^m \eps^{j-1}U^{j}(t,x, \cdot)$ and derivatives of $f'$, and $a$ by $\psi$, $\phi$, $\omega$ 
 and their derivatives.

 \br\label{iso_rmk}
 The careful choice of time-varying coordinates $y$ is made here to avoid cross diffusion terms in 
 the $(z,y)$ representation, thus preserving up to absorbable errors and transverse drift
 the isotropic form of the equations and allowing the reduction of symmetrizer calculations
 to the one-dimensional case.
 In treating the case of general, anisotropic diffusion, there would be no advantage to such coordinates,
 and no harm to choosing a constant coordinate frame $y$ analogously as in \cite{MZ}.
 \er

 We now mimic \cite{MZ}\footnote{Compare principal terms of \eqref{system} with the equivalent \cite[eq. (2.13)]{MZ}.} 
 and write \eqref{eq3} as a system
 \begin{equation} \label{system}
	 \D_z  V  = \Big(\frac{1}{\eps k }  M + O(1) +o(\eps |\D_y^2)|) +O(\eps|\D_y|) \Big)  V  + E , 
 \qquad V =  \begin{pmatrix} v \\ \eps  k \D_z v  \end{pmatrix} , 
 \qquad  E   =  \begin{pmatrix} 0  \\ \frac 1k    e^\eps   \end{pmatrix}  , 
 \end{equation} 
 where
 \begin{equation}
 M (t, x, \eps \D_t) =      
	 \begin{pmatrix} 0 & 1 \\ \eps \D_t + 
 \eps \sum_j c_j \D_{y_j} 
	 + G (k, z / \eps )- \eps^2 \D_y^2   &    \omega(k)   \end{pmatrix}. 
 \end{equation} 
We will perform a semi-classical pseudo-differential analysis in $t$ and $y$, and replace 
 $\eps \D_t$ and $\eps \D_y$ by their symbols $\lambda=i \tau$ and $i\eta$, $\eta=(\eta_2,\dots,\eta_d)$.
See Appendix \ref{s:paradiff_app} or \cite[\S3.1]{MZ} for a brief description of the relevant tools, phrased in
the paradifferential calculus of Bony \cite{B}.
 Moreover,   we consider \eqref{system} as an evolution equation in the fast variable
 $\theta:=z/\eps$, 
 and  this yields to consider the system
 \be\label{exact}
 \D_\theta \cV  =  \Big( \cM +O(1) +o(\eps |\eta|^2) +O(\eps|\eta|) \Big) \cV + \cE ,
 \qquad \theta \in \R,
 \ee
with
  \begin{equation}\label{Mexact}
\cM  (k, \theta  , \lambda, \eta)  =
\frac{1}{k}  \begin{pmatrix} 0 & 1 \\ (i\tau +
 \eps \sum_j c_j i \eta_j + \eps|\eta|^2)   + G (k, \theta)   &    \omega (k)   \end{pmatrix},
 \end{equation} 
 or, dropping error terms,
 \begin{equation}
 \label{modeq}
 \D_\theta \cV  =  \cM \cV + \cE.
 \end{equation} 

 Note that $\theta$ here {\it varies on the line $\R$}, as we are not imposing $\theta$-periodicity
 on $V$; note also that here $k=\nabla_x \Psi=\nabla_x \psi +O(\eps)$, allowing an $O(\eps)$ perturbation of
 the prescription of Section \ref{s:asymptotic}.
 
Our goal is to prove the following basic $L^2$ estimates for the solutions of \eqref{system}: 

\begin{prop}\label{Vprop}
There is $\gamma _0 \ge 0$ such that  for $\gamma + \eps|\eta|^2 \ge \gamma_0$,
solutions $V$ of \eqref{system} satisfy
\begin{equation}\label{Vest}
\big\| V \big\|_{L^2_\gamma}  \lesssim   \big\| E \big\|_{L^2_\gamma} ,
\end{equation} 
where $L^2_\gamma = e^{ - \gamma t} L^2 (\RR_t \times \RR_x)$ and $V$, $E$ are supported
on a sufficiently small neighborhood of $(\underline x, \underline t)$.
\end{prop} 

Instead of $V$ we introduce $\tilde V = e^{ - \gamma t} V$, and dropping the tildes we are reduced to proving 
\begin{equation}
\label{mainest}
\gamma \big\| V \big\|_{L^2}  \lesssim    \big\| E \big\|_{L^2} 
\end{equation} 
for the solutions of \eqref{exact}, where now, setting $ \lambda:= \gamma + i\tau,$
 \begin{equation}
 M (t, x,  \lambda, \eta) =   
\begin{pmatrix} 0 & 1 \\ (\lambda + \sum_j c_j i \eta_j +
\eps|\eta|^2) + G (k, z / \eps )   &    \omega(k)   \end{pmatrix} .
\end{equation} 

\subsection{Paradifferential calculus and proof of the main estimates}\label{s:paradiff}
We split $V$ into high and low (and medium) frequencies 
\begin{equation}
	V_l  = \chi \big( \zeta(\eps D_t, \eps \gamma, \eps D_y)\big) V, 
	\qquad V_h = \Big(1 - \chi \big( \zeta (\eps D_t, \eps \gamma, \eps D_y)\big)\Big) V 
\end{equation}
where $\zeta:=|\gamma + i(\D_t + \sum_j c_j i \eta_j) | + \eps |\D_y|^2$  and
$\chi$ is a $C^\infty$ cutoff function  equal to $1$ on a large ball $B(0,R)$ to be chosen later on, and
zero outside $B(0,R+1)$.
The commutator of $\chi (\eps D_t,\eps \gamma, D_y)$ with the equation is 
$O(1) V$ (see Propositions \ref{j11} and \ref{j14a}, Appendix \ref{s:paradiff_app}),
hence can be absorbed by choosing $\gamma_0$ large enough.
So we are reduced to proving the estimates for $V_l$ and $V_h$ separately. 

\subsection{Low and medium frequencies} \label{s:lf}
For $\zeta:=|\gamma + i(\D_t + \sum_j c_j i \eta_j) |+ \eps |\D_y|^2$ 
in a bounded region we use the following reduction. 
 
\begin{prop}[Floquet's Lemma] \label{floquet}
There exists an invertible smooth  periodic  matrix-valued function $\cW (k, \theta, \lambda,\eta)$, 
such that the change of coordinates $\cV_1 = \cW \cV $ reduces \eqref{modeq} to 
\begin{equation}\label{e:ccv}
	\D_\theta \cV_1 =\cM_1 \cV_1 + \cE_1 
\end{equation}
 where $\cM_1 = \cM_1 (k, \lambda,\eta)$ is independent of $\theta$. Equivalently, $\cW$ solves 
 \begin{equation}
 \label{conj}
 \D_\theta \cW  + \cW \cM = \cM_1 \cW.  
 \end{equation}
\end{prop}
\begin{proof}
Let $\cX ( \theta)$ be the fundamental matrix of the system
 \eqref{modeq}, 
 $$
 \D_\theta \cX  = \cM \cX , \quad \cX (0) = \Id \qquad 
 \Rightarrow \qquad 
 \D_\theta \cX^{-1}   = -  \cX^{-1}  \cM , \quad \cX^{-1}  (0) = \Id.  
 $$
Consider a  constant matrix $\cM_1$ and $\cW = e^{  \theta\cM_1 } \cX^{-1}  (\theta)$. Then 
$$
\D_\theta  \cW  =  \cM_1 \cW - \cW\cM  . 
$$
Thus $\cW$ conjugates the system \eqref{modeq} to the constant coefficient system \eqref{e:ccv}. Moreover, 
$\cW$ is periodic if and only if $\cW (1) = \Id$, that is 
 \begin{equation}
 e^{\cM_1} = \cX (1). 
 \end{equation} 
Because $\cX(1)$ is invertible, we can choose a logarithm $\cM_1 = \ln \cX (1)$. 
\end{proof}

 Accordingly, we make the change of unknowns
 \begin{equation}
 V_1 (t, z)  = \cW (k,  z / \eps, \eps \D_t +  \eps \gamma ) \tilde \chi (\eps \D_t, \eps \gamma) V_l(t, z)
 \end{equation}
 where we now consider $k$ as a given function of the variables $t$, $y$, and $z$ and $\tilde \chi = 1$ 
 on the support of $\chi$.  The symbolic calculus shows that the commutators are $O(1) V_l $ and 
  \begin{equation}
 \label{system2}
	  \D_z  V_1  = \frac{1}{\eps  }  \Big(\cM_1 (k, \eps \D_t+  \eps \gamma, \eps \D_y) \Big)V_1 + E_1    
 \end{equation} 
 with $E_1 = \Big( O(1) +o(\eps |\eta|^2) +O(\eps|\eta|) \Big)  V_1   + O(1) E$. 
 Next we use the method of symmetrizers. 
 
 \begin{prop} \label{symm} For $ \lambda + i\sum_j c_j  \eta_j) + |\eta|^2$  bounded, $\lambda=\gamma +i\tau$, 
	there exist locally smooth symmetrizers for $\cM_1$, that is, matrices 
$\cS (k,  \lambda,\eta ) $, $C^1$ in $(\lambda,k,\eta)$ such that $ \cS = \cS^*$, $|\cS|$ uniformly bounded, and
\begin{equation}\label{symmeq}
\Re \cS (k,  \lambda, \eta)  \cM_1 (k, \lambda, \eta)    =  (\gamma+ |\tau + \sum_j c_j  \eta_j |^2+ |\eta|^2) 
\Gamma (k, \lambda,\eta), \qquad \Gamma \ge \Id.  
 \end{equation}
 \end{prop}

\begin{cor}\label{lfcor}
On a neighborhood of $(\underline {x}, \underline{t})$ such that error term $o(\eps |\eta|^2)$ is sufficiently small
	compared to $\eps |\eta|^2$, $V_l$ satisfies 
	\begin{equation}\label{preVest}
	 (\gamma + \eps |\tau + \sum_j c_j  \eta_j |^2+ \eps |\eta|^2) 
		\big\| V_l  \big\|_{L^2}  \lesssim   \big\| E \big\|_{L^2} + \big\| V_l \big\|_{L^2}. 
\end{equation} 
\end{cor}  
 \begin{proof}
 Use the energy balance
	 \begin{equation}\label{ebal}
\Re  \big(   \cS E_1 , V_1\big)_{L^2}  = \frac{1}{\eps} \Re  \big(   \cS \cM_1 V_1 , V_1\big)_{L^2} 
	 + o(1)\frac{1}{\eps} \|V_1 \|_{L^2} 
+    \big(  K V_1 , V_1\big)_{L^2}
 \end{equation} 
 $$
 K = - \mez \D_z \cS +  \eps^{-1} 
 \Re op (\cS \cM_1)  + \eps^{-1}  \Re \Big(  op( \cS) op ( \cM_1 ) - op (\cS \cM_1 ) \Big) 
 $$
 where $op$ denotes the semiclassical quantification of symbols. By the symbolic calculus, the last 
 term is $O(1)$ and the second one is $\eps^{-1} op (\Re \cS \cM_1) +  O(1)$. 
	 Finally, by \eqref{symmeq} applied to \eqref{system2},
	 $$
	 \eps^{-1} op (\Re \cS \cM_1) = (\gamma+\eps|\tau + \sum_j c_j  \eta_j |^2 + \eps|\eta|^2) op( \Gamma),$$
	 which proves that the right-hand side of \eqref{ebal} is 
$ \gtrsim   (\gamma + \eps |\tau + \sum_j c_j  \eta_j |^2+ \eps |\eta|^2) 
	 \big\| V_1 \big\|^2_{L^2}   + O ( \big\| V_1  \big\|_{L^2} ).  $
Meanwhile, the left-hand side is $\lesssim 
	 \big\| E \big\|^2_{L^2}   + 
	 \Big( O(1) +o(\eps |\eta|^2) +O(\eps|\eta|) \Big) \big\| V_1  \big\|_{L^2}^2 $, where
	 $o(\eps |\eta|^2)\ll \eps |\eta|^2$ and $O(\eps |\eta|)= \sqrt{\eps}\sqrt{O(\eps |\eta|^2)}=o(1)$.
	 Combining, and absorbing error terms, we obtain the result.
For details of the pseudodifferential computations used here, see Appendix \ref{s:paradiff_app} or \cite[\S3.1]{MZ}.
 \end{proof} 

From \eqref{preVest}, the estimate \eqref{Vest} follows for $\gamma$ large enough, completing the proof
of Proposition \ref{Vprop} for low and medium frequencies.

\subsubsection{Proof of Proposition~\ref{symm}} 
It remains to establish existence of symmetrizers, Proposition~\ref{symm},
for the averaged coefficient matrix $\cM_1$ of \eqref{e:ccv}.
To this end, we first deduce the eigenstructure of $\cM_1$ from Assumption \ref{ass3}, via the
{\it periodic Evans function} of Gardner \cite{G}, which, in the coordinates of \eqref{e:ccv}, takes
the simple form $D(\lambda, \xi,\eta)= \det \big( e^{2\pi \cM_1(k,\lambda, \eta)}- e^{2\pi i\xi} \big).$
Evidently analytic with respect to $\cM_1$, $\xi$, the Evans function has the
fundamental property \cite{G,Z2} that zeros of $D(\cdot, \xi,\eta)$ agree in location {\it and multiplicity}
with eigenvalues of the Bloch-Fourier operator $\cL_{\xi,\eta}$ of \eqref{B-F}.

Observing as in Remark \ref{multirmk} that
$\lambda$, $\eta$ enter $\cM_1$ only in the combinations $\tilde \lambda:=(\lambda + |\eta|^2)$
and $\tilde \tau:= \tau + \sum_j c_j  \eta_j$, we see that it is sufficient
to treat the 1-D case $\eta\equiv 0$.
For simplicity, {\it take $\eta\equiv 0$ from now on, and consider the 1-D Evans function}
\be\label{1DEvans}
D(\lambda, \xi)= \det \big( e^{2\pi \cM_1(k,\lambda )}- e^{2\pi i\xi} \big),
\ee
$\cM_1(k,\lambda)=\cM_1(k,\lambda,0)$, and its relation to the Bloch operator $\cL_\xi=\cL_{\xi,0}$.

By the spectral mapping theorem, zeros $\lambda$ of $D(\cdot, \xi)$ correspond to pure imaginary (matrix) eigenvalues
$\mu= i\xi$, (mod $2\pi$) of $\cM_1(k,\lambda)$.
But, by the properties of the Evans function, these also correspond to (operator) eigenvalues $\lambda$ of $\cL_\xi$.
Thus, by Assumption \ref{ass3}, $\cM_1(k,\lambda)$ has no pure imaginary eigenvalues for $\gamma=\Re \lambda \geq 0$,
except for the eigenvalue $\mu=0$ (mod $2\pi$) at $\lambda=0$, which, by choice of the logarithm function 
in the proof of Proposition \ref{floquet}, may be normalized as $\mu=0$.

\medskip

({\it Medium frequencies.}) For medium frequencies, $1/R\leq |\lambda|\leq R$, we have by continuity of spectra and 
compactness in $(\lambda,k)$ that $\cM_1$ has a uniform spectral gap 
$\Re \sigma(\cM_1(\lambda, k))\geq c_0>0$, whence
there exist smooth coordinate transformations $T(\lambda,k)$
reducing $M_1$ to form 
$$
T\cM_1 T^{-1}(\lambda,k)=:\tilde{\cM}_1= \begin{pmatrix} P_+& 0\\0 & P_-\end{pmatrix}(\lambda,k),
$$
where $P_+, -P-\gtrsim 1$. Since 
$\gamma +|\tau| \sim |\lambda| \lesssim 1$ by assumption, hence also $\gamma +|\tau|^2\lesssim 1$, we thus have
$$
P_+, -P-\gtrsim 
\gamma +|\tau|^2 .
$$
By Lyapunov's Lemma, there exist $S_+$, $S_-$ symmetric with $\Re (SP) \gtrsim \gamma + |\tau|^2 $, hence 
$ \tilde{\cS}:=  \begin{pmatrix} S_+& 0\\0 & S_-\end{pmatrix}$ is a symmetrizer for $\tilde{\cM}_1$, and
	$ \cS:=|T^{-1}|^2 T^*\tilde{\cS}T $ is a symmetrizer for $\cM_1$, with
$$
\Re(\cS \cM_1)=|T^{-1}|^2 T^* \big( \Re (\tilde{\cS}\tilde{\cM}_1)\big) T = |T^{-1}|^2
(\gamma+|\tau|^2) T^* \tilde {\Gamma}T=: (\gamma +|\tau|^2)\Gamma,
$$
where $\tilde{\Gamma}\geq \Id$ and thus $\Gamma\geq |T^{-1}|^2 (T^*\tilde {\Gamma}T)\geq \Id$.
We note in passing that this argument demonstrates the important observation of Kreiss
\cite{K} that the property of existence of a symmetrizer to be invariant
under smooth coordinate transformations, a fact we shall use freely below.

\medskip

({\it Low frequencies.}) We now come to the key, low-frequency case $|\lambda|\leq 1/R$, $R>0$ sufficiently large,
where lies the main difficulty of the symmetrizer construction.
Here, we have by Assumption \ref{ass3} that the eigenvalues of $\cM_1$ split into a strongly
stable subset with real part strictly negative, a strongly unstable subset with real part strictly positive, 
and a single small eigenvalue $\mu_*$ that is uniformly spectrally separated from both,
associated with the ``neutral stability'' curve 
$$
\{\lambda:\, \lambda=\lambda_*(\xi)\},
$$
where $\lambda_*(\xi)$ 
is the eigenvalue of $\cL_\xi$ bifurcating from the simple ``translational'' eigenvalue $\lambda=0$ of $\cL_0$.

By spectral separation of these three groups of eigenvalues, there exists a smooth coordinate transformation 
$T(\lambda,k)$ transforming $\cM_1$ to block-tridiagonal form
$$
\tilde {\cM}_1=\begin{pmatrix} P_+& 0&0\\0 & P_- &0 \\ 0 & 0 & m \end{pmatrix}(\lambda,k),
$$
where $P_+, -P-\gtrsim 1 \gtrsim \gamma $. Taking $\cS_\pm$ as in the previous case, we see that it is sufficient
to find a symmetrizer $s$ for $m$, in which case 
$ \cS:= \blockdiag \{  S_+, S_-, s\}$ is a symmetrizer for $\tilde{\cM}_1$, and we are done.
We are thus reduced to constructing a symmetrizer for the block $m$ associated with the small eigenvalue $\mu_*$,
considered as a separate analytic function $m(\lambda,k)= m_0 + \lambda m_1 + \dots$.
By a further coordinate transformation, we may reduce to the case that $m_0=m(0)$ is in Jordan form.
The treatment of this neutral block hinges on the following linear-agebraic observation.

\begin{lemma}\label{blocklem}
Let $d(\lambda,\xi) := \det \big( e^{2\pi m(\lambda )}- e^{2\pi i\xi} \big)$ in a neighborhood of
	$\lambda=0$, with $m : \C \to \C^{r\times r}$ in $C^s$, $s\geq 1$, $m(0)$ a nilpotent standard Jordan form,
and $\D_\lambda d(0,0)\neq 0$, and let $\lambda(\xi)\in C^1$ be the unique local function 
defined implicitly by $d(\lambda(\xi),\xi)=0$.
Then, $m(0)$ consists of a single Jordan block, and $m_1:=m'(0)$ has nonvanishing $r$-$1$ entry $\alpha=-\D_\lambda
	d(0,0)$.
Moreover, for $j\leq s$, 
\be\label{keyrel}
	\frac{d^j \lambda(0)}{j!}=\begin{cases}
0 & \hbox{\rm for $ j<r$},\\
 (i)^r/\alpha & \hbox{\rm for $j=r$}.
	\end{cases}
	\ee
\end{lemma}

\begin{proof}
	In the case that $m(0)=J$ is a single $r\times r$ Jordan block, we find by Taylor expansion that
	$$
	e^{2\pi m(\lambda)}- e^{2\pi i\xi}= 
	2\pi \lambda( J + m_1) -  (e^{2\pi i\xi}-1) + O(\lambda^2),
	$$
	whence, by direct computation, $d(0,0)=0$ and $\D_\lambda d(0,0)= -(m_1)_{r1}$.
	In the general case, $d$ decomposes into the products of the Evans functions for the different Jordan blocks
	of $m(0)$, hence $\D_\lambda d(0,0)= 0$ if there were more than one block, and so, by contradiction,
	there is only one.  This establishes the first assertion.
	For $\lambda, \xi$ small, we may expand $(e^{2\pi i\xi}-1)$ as well, to obtain
	$$
	0=d(\lambda(\xi),\xi)/2\pi= \det \big(  \lambda (J+ m_1) -  (i\xi + O(|\xi|)^2)\Id + O(|\lambda|^2) \big)
	= -(m_1)_{r1}\lambda - (i\xi)^r + O(|\lambda|^2 + |\xi|^{r+1}), 
	$$
	from which we may obtain the second assertion by implicit differentiation.
\end{proof}

\begin{cor}\label{c0cor}
Under Assumption \eqref{ass3}, either (i) $\omega'(k)\neq 0$, $r=1$,
$\alpha=-1/\omega'(k)$ is real, $\lambda_*(\xi)= -i\xi \omega'(k)-b\xi^2 +O(\xi^3)$, and
$m(\lambda,k)= -\lambda/\omega'(k) + b\lambda^2/\omega'(k)^3 + O(\lambda^3)$ with $b$ positive real, or
(ii) $\omega'(k)= 0$, $r=2$, $\alpha$ is positive real, and $\lambda_*(\xi,k)=-\frac{1}{\alpha} \xi^2 + O(\xi^3)$.
\end{cor}

\begin{proof}
Recall \cite{Se,DSSS,JNRZ2,SSSU} that the neutral spectral curve $\lambda_*$ has Taylor expansion 
$$
\lambda_*(\xi)= -i\omega'(k)\xi - b(k)\xi^2 + \dots
$$
about $\xi=0$, $b$ real, corresponding to the second order formal Whitham expansion \cite{W}
\be\label{wexp}
k_t + \D_x \omega(k)=  \D_x(b(k) \D_x k).
\ee
The result then follows by Lemma \ref{blocklem}, together with the observation, above, 
that $\lambda_*$ is determined by the reduced Evans function $d$ associated with the neutral Jordan block
	of $\cM_1$ at $\lambda=0$, followed in case (i) by inversion of relation 
	$\lambda =\lambda_*(\xi,k)$
	to get the Taylor expansion of $i\xi= m(\lambda,k)$.
\end{proof}

\br\label{ozrmk}
Note that the neutral eigenvalue $\mu(0)=0$ of $\cM_1(0)$ may have higher multiplicity
even though $0$ is a simple root of $\cL_0$; as a consequence, $\mu(\lambda)$ may be singular at $\lambda=0$
even though $\cM_1$ is analytic.  In particular (cf. \cite{OZ1}), when $\omega'(k)=0$, $\mu$ has 
a {square-root singularity}, $\mu(\lambda)\approx c\sqrt{\lambda}$.
We note further that the apparently degenerate case $\omega'(k)\equiv 0$ is in fact quite common,
occurring generically for stationary solutions $\omega(k)=0$, by reflection-invariance $x\to -x$ of \eqref{eq1}
\cite{DSSS}.
\er

We are now ready to construct symmetrizers for the neutral block $m(\lambda,k)$. In case (i), $\omega'(k)\neq 0$,
$m(\lambda,k)= -\lambda/\omega'(k) + b\lambda^2/\omega'(k)^3 + O(\lambda^3)$ is scalar by Corollary \ref{c0cor},
with $b$ positive real, 
hence $ s(\lambda,k)= -\omega'(k) $ is a symmetrizer, smooth in $(\lambda,k)$, with 
$$
\Re(sm)(\lambda,k)= \Re (\lambda + b\lambda^2/\omega'(k)^2)= \gamma + (b/\omega'(k)^2)\Re(\lambda^2) + 
O(\lambda^3)\gtrsim \gamma + |\tau|^2.
$$
In case (ii), $\omega'(k)=0$ at $k=k_*$,
$$
m(\lambda,k)=
\begin{pmatrix} 0 & 1\\ 0 & 0 \end{pmatrix}
	+ \begin{pmatrix}  a(\kappa,\lambda)\lambda  & b(\kappa,\lambda)\lambda\\ 
	c_0 \lambda + c(\kappa,\lambda)\lambda^2 + e_0(\kappa) \kappa \lambda  & d(\kappa,\lambda)\lambda \end{pmatrix}
+ \kappa n(\kappa),
$$
by Corollary \ref{c0cor}, $ \kappa:=k-k_*$, $c_0$ constant
real and positive, without loss of generality (rescaling $\lambda$) $c_0=1$.
This can be recognized as a variant of the case of a {\it glancing mode} of order $2$ arising in the theory of hyperbolic
boundary-value problems, for which we may use a construction like that of Kreiss \cite{K} in the hyperbolic
boundary-value setting to obtain a smooth symmetrizer $s(\lambda,k)$. 
We carry out this more complicated construction separately, in Appendix \ref{s:kcon}.

In either case, the constructed neutral-block symmetrizer $s$ yields a full block-diagonal local symmetrizer
for $\cM_1$ that is smooth in $(k,\lambda)$, yielding the desired estimate \eqref{Vest}.
This completes the proof of Proposition \ref{Vprop} for low and medium frequencies.

\subsection{High frequencies}\label{s:hf} 
For $\zeta:=|\gamma + i\D_t|+ \eps |\D_y|^2$ sufficiently large, we may proceed by the argument already
established in \cite{MZ}.
Namely, we may construct a symmetrizer for the principal-part 
symbol of the original, periodic in $z$ system, by a simplified version of \cite[Lemma 2.14]{MZ}
(establishing property (ii) of the reference and ignoring properties (i) and (iii)),
applying a periodic block-diagonalizing transformation at each point $z$ and noticing that commutator errors absorb,
reducing $M$ to form $\begin{pmatrix} P_+& 0\\0 & P_-\end{pmatrix}$, where $P_+, -P-\gtrsim \gamma + \eps |\eta|^2$.
By Lyapunov's Lemma, there exist $S_+$, $S_-$ symmetric such that $\Re (SP) \gtrsim \gamma + \eps|\eta|^2$, hence 
$$
\cS:= \begin{pmatrix} S_+& 0\\0 & S_-\end{pmatrix}
$$
serves as a symmetrizer giving $\Re (\cS M)\gtrsim \sqrt{\gamma + \eps |\eta|^2}$.
The result for the full system is then obtained, similarly as in the proof of Corollary \ref{lfcor}
by pseudodifferential estimates showing that commutators and other errors absorb,
from which \eqref{Vest} immediately follows.
The latter computations are carried out using the semiclassical parabolic paradifferential calculus 
described in \cite[\S 32]{MZ}.  For details, see \cite{MZ}.
This completes the proof of Proposition \ref{Vprop}.

\br\label{scalermk}
Note, as in \cite{MZ}, the essentially different scalings in bounded- vs. high-frequency regimes, 
as evidenced by the different bounds $\Re (\cS M)\gtrsim \gamma + \eps|\tau|^2 + \eps |\eta|^2$ vs.  
$ \Re (\cS M)\gtrsim \sqrt{\gamma + \eps |\eta|^2}. $\footnote{
The latter may be sharpened slightly to $\Re (\cS M)\gtrsim \sqrt{\gamma + |\tau|+ \eps |\eta|^2}$, though
we do not show it here.}
\er

\medskip

From Proposition \ref{Vprop}, we readily obtain our final linear bounds.

\begin{proof}[Proof of Theorem \ref{main2}]
	From \eqref{Vest} of Proposition \ref{Vprop} and the definition of $V$, $E$ in \eqref{system}, 
	we obtain for solutions $v$, $h$ of \eqref{linprob}
	supported in a sufficiently small neighborhood of $(\underline x, \underline t)$
	and $\gamma>0$ sufficiently large the estimate 
	$ \gamma \|e^{-\gamma t} v \|_{\cH^{1}_\eps} \lesssim  \|e^{-\gamma t} h \|_{L^2_\eps}.  $
	This may be extended to general $v$, $h$ by a partition of unity argument as in the
	proof of Proposition~5.1 in \cite[\S 5.2]{MZ}, as we now describe, 

	Namely, we first observe, by the property 
	(\eqref{zerolim}) that $\D_x^2 \psi \to 0$ as $|x|\to \infty$
	together with 
	$$
	k:=\D_x \Psi =\D_x \psi + O(\eps), 
	$$
	that for $\eps$ sufficiently small and $R$ sufficiently large,
	we may obtain the same estimates for $v$, $h$ supported on any neighborhood
	lying outside $C_R:=\{(t,x): \, |x|\leq R\}$ and of diameter $\leq 1/R$.
	For, the size of allowable neighborhoods is determined by required smallness of $o(\eps|\eta|^2)$ terms
	relative to $\eps |\eta|^2$, coming from change of coordinates of the Laplacian diffusion terms to
	the $(z,y)$ frame, $z=\Psi$, specifically, error terms arising from nonconstancy of $k=\nabla_x \Psi$
	that are controlled by 
	$$
	\|\D_x^2\Psi\|_{L^\infty}\sim \|\D_x^2\psi\|_{L^\infty} + \eps,
	$$
	together with existence of a smoothly varying frame $(z,y)$,
	which holds so long as variation of $k=\nabla_x \Psi$ is small, in particular for diameter times
	$\|\D_x^2 \Psi\|_{L^\infty}$ small.

	We may thus cover $C_R^\complement$ by a countable collection of identical translates 
	$\Omega_j$ on which the estimates are satisfied, for which
	each point $(t,x)\in C_R^\complement$ lies in at most a fixed finite number $N$ of the $\Omega_j$.
	(For example, we may achieve $N=2^d$ by tiling $C_R^\complement$ with identical rectangular tiles
	$R_j$, then taking $\Omega_j:= R_j \oplus B(0,\delta)$ for $\delta>0$ sufficiently small.) 
	Covering the compact set $C_R$ with finitely many more $\Omega_j$, we obtain a countable cover
	$\{\Omega_j\}$ of $\R^d\times [0,T]$ for which all but finitely many are identical translates, 
	each point $(t,x) \in \R^d\times [0,T]$ lies in at most $N_1$ of the $\Omega_j$,
	and the estimates are satisfied for $v$, $h$ supported in $\Omega_j$.

	Defining a partition of unity $\sum_j \chi_j$ subordinate to $\{\Omega_j\}$, we have that the estimate
	holds for each $v_j:=\chi_j v$.  Moreover, by construction $\sup |\D \chi|, \, \sup |\D^2\chi|\leq C$
	for some fixed $C>0$.
	Computing 
	$$
	\begin{aligned}
		\eps h_j:&= \eps \cP_{u^{\eps,m}}v_j = \eps \big( \chi_j h + 
	O(|\D \chi_j|) (| v|+ |\eps \D_x v|)+O(|\D^2_x \chi_j|(\eps)|v| \big)\\
		& =
	\eps \big( \chi_j h + O(| v|+ \eps |\D_x v|) \big) 
	\end{aligned}
	$$
	and using the fact that each $(t,x)$ lies in at most $N_1$ of the $\Omega_j$, 
	we find that the sum over $j$ of the $L^2$ norm of commutator terms
	$O(|\D \chi_j|) (|v|+ |\eps \D_x v|)+O(|\D^2_x \chi_j|(\eps)|v|$
	is $\lesssim  N_1 \|v\|_{\cH^1_\eps}\lesssim   \|v\|_{\cH^1_\eps}$, 
		hence absorbs in the left-hand side of the error estimate,
	giving the result for $s=0$.  See the discussion of \cite[p. 57]{MZ} for a similar argument in 
	the hyperbolic-parabolic boundary-layer case.

	Derivative estimates $s+1\geq 1$ then follow by a standard induction, 
	differentiating the equation and absorbing lower-order
	commutator terms using the estimates obtained previously in $\cH^s_\eps$, to yield
	$\gamma \|e^{-\gamma t} v \|_{\cH^{s+1}_\eps} \lesssim  \|e^{-\gamma t} h \|_{H^s_\eps}$
	for $s$ in the range specified for Theorem \ref{main2}.
	The desired estimate \eqref{linbd} then follows by the observation that
	$e^{-\gamma t}\sim 1$ for $t$ on the bounded domain $[0,T]$.
\end{proof}



\section{Nonlinear convergence}\label{s:nonlin}
With linear estimates in hand, nonlinear validity now follows by a standard contraction mapping argument
together with some care in dealing with initial-/boundary-layers in time variable $t$.

\subsection{Prepared data}\label{s:prepared}
For simplicity of exposition, and because the argument seems of interest in its own right, we
first treat the easier case of ``prepared data,'' seeking an exact solution {\it near} the approximate
solution $u^{\eps,m}$, but not necessarily agreeing at initial time $t=0$.

\begin{proposition}\label{prepprop}
Under the assumptions of Corollary \ref{main3}, for $\eps>0$ sufficiently small, there exists an exact solution 
$u\in \cH^{s+1}_\eps$ of \eqref{eq1} on $[0,T]$, satisfying	
\be\label{rembd2}
\|u-u^{\eps,m}\|_{\cH^{s+1}_\eps} + \eps^{1/2} \|u-u^{\eps,m}\|_{C^{s -[d/2]}_\eps} 
\lesssim \eps^{m}\|\D_{t,x}^2\psi\|_{H^{s+ 2(m+1)}}.
\ee
\end{proposition}

\begin{proof} By the assumptions on $\psi$, and standard smooth hyperbolic theory, we may extend $\psi$
	to a slightly enlarged time interval $[-\delta, T]$ on which it satisfies (a multiple of)
	the same bounds.
	Thus, for $0<\eps<\delta$, we may extend the residual equation \eqref{eq2intro}
	from $t\in [0,T]$ to $t\in [-\delta, T]$, via
	 \be\label{modeq2} 
	 \eps \cP_{u^{\eps,m}}v^{\eps,m} =  \eps \chi^\eps (t) e^{\eps,m} ,
	 \ee
	 where $\chi^\eps(t):= \chi(-t/\eps)$, with $\chi(z)$ a smooth cutoff function
	 vanishing for $z\geq 1$ and $1$ at to $z=0$. 
	 Evidently, \eqref{modeq2} agrees with \eqref{eq2intro} on the original time domain $[0,T]$,
	 and $\chi^\eps(t)e^{\eps,m}$ vanishes for $t\leq -\eps$;
	 moreover, the extension and cutoff functions do not change the bounds on the remainder, nor on
	 the quadratic-order function $Q$.
	 Thus, it is sufficient to solve \eqref{modeq2} and then restrict to $[0,T]$.
	
	Defining now $\eps^m \bar v^{\eps,m}:= v^{\eps,m}$ and
	$\bar Q(u^{\eps,m}, \bar v) :=  \eps^{-2m} Q(u^{\eps,m}, \eps^m \bar v)$, 
	and inverting \eqref{eq2intro}, we may express 
$\bar v^{\eps,m}$ as a solution of the fixed-point problem
	\be\label{barmeq}
	\bar v=\cT \bar v:=  \cP_{u^{\eps,m}}^{-1} \big( -  \eps^{-(m+1)}R^{\eps,m}  
	+ \eps^{m-1} \bar Q(u^{\eps,m}, \bar v) \big)
	\ee
	on $t\in [-\delta, T]$. (Here, we are using Sobolev embedding to see that 
	$\|\bar v\|_{L^\infty}\lesssim \eps^{-1/2}$, hence $\|\eps^m\bar v\|_{L^\infty}$ by $m\geq 2$
	is $O(\eps^{3/2})$, thus small.) 
	Recalling that $|\cP_{u^{\eps,m}}^{-1}|_{\cH^{s}_\eps}$ and  
	$\|\eps^{-(m+1)}R^{\eps,m}\|_{\cH^s_\eps}$ are uniformly bounded, and $\bar Q$ smooth and quadratic order
	in $\bar v$, so that the Lipschitz norm of $\bar Q$ with respect to $\bar v$ is
	is $\lesssim \|\bar v\|_{L^\infty}\lesssim \eps^{-1/2}$ 
	by Sobolev embedding/Moser's inequality, we find that $\cT$ is a contraction mapping with Lipschitz constant
	$\lesssim \eps^{m-3/2}\lesssim \eps^{1/2}=o(1)$ for $m\geq 2$ on a ball of radius 
	$\lesssim \|\eps^{-(m+1)}R^{\eps,m}\|_{\cH^s_\eps}$,  yielding
	a unique solution $ v^{m,\eps}\in \cH^s_\eps$ with $v^{m,\eps}=0$ for $t\leq -\eps$ and
	$
	\| v \|_{\cH^s_\eps} \lesssim \eps^m \|\eps^{-(m+1)}R^{\eps,m}\|_{\cH^s_\eps} 
	\lesssim \eps^{m}\|\D_{t,x}^2\psi\|_{H^{s+ 2(m+1)}},
	$
	by \eqref{asHbds}(ii). 
	Applying \eqref{linbd} to the original equation \eqref{eq2intro}, with $h:=e^{\eps,m}$,
	we obtain \eqref{rembd2}.
\end{proof}

\br\label{prepared}
Evidently, the solution $u$ obtained by this argument is not unique, nor does the argument show
that the exact ``prescribed data'' solution $u|_{t=0}=u^{\eps,m}|_{t=0}$ remains close, or even is
defined on an $O(1)$ interval of time,
its guaranteed time of existence $0\leq t\lesssim \eps$ being given by
well-posedness of the unscaled system $\D_t u+ f(u)= \Delta_x u$. 
Note that in the extension of \eqref{eq2intro} to $[-\delta, T+\delta]$ we made use of reversibility 
of the hyperbolic equation for $\psi$.
For an irreversible, e.g., diffusive modulation equation, one could solve forward to extend from
$[0,T]$ to $[0, T+ 2\eps]$, then use cutoffs to obtain a nearby solution for $t\in [\eps, T+\eps]$.
Alternatively, one might restrict to {\it analytic data} for which the diffusive equation may be solved in reverse time,
giving a result on $[0,T]$.
\er

\subsection{Prescribed data}\label{s:prescribed}
To treat the prescribed data problem $u|_{t=0}= u^{\eps,m}|_{t=0}$, or $v|_{t=0}=0$, 
we first examine the initial layer resulting from the mismatch between 
$u$ and $u^{\eps,m}$ equations, i.e., from forcing $R^{m,\eps}\not \equiv 0$ in $v$-equation
$\eps \cP_{u^{\eps,m}}v= \eps e^{\eps,m}$  of the introduction 
(\eqref{eq2intro}--\eqref{cPdef}):
\be\label{veq}
\eps  \D_t v   +  g^\eps v     -   \eps^2 \Delta_x v= -  R^{\eps,m}  + Q(u^{\eps,m}, v^{\eps,m}).
\ee

\bl\label{fastlem}
Under the assumptions of Corollary \ref{main3}, for 
$\|w_0\|_{\cH^s_\eps(\R^{d})}\lesssim \eps^{m+1}$, $m\geq 2$, and $\eps>0$ sufficiently small,
there exists a unique solution 
$w\in \cH^s_\eps(\R^d\times [0,\eps])$ of \eqref{veq} with data $w|_{t=0}=w_0$, satisfying
\be\label{wbd}
\|w\|_{\cH^{s+1}_\eps} \lesssim \|w_0\|_{\cH^s_\eps(\R^d)} + \|R^{\eps,m}\|_{\cH^s_\eps(\R^d\times [0,T])}. 
\ee
\el

\begin{proof}
	The rescaling $(t,x)\to (t/\eps, x/\eps)$ converts \eqref{veq} to
$$
\D_t v   +  g^\eps v     -   \Delta_x v= -  R^{\eps,m}  + Q(u^{\eps,m}, v^{\eps,m}),
$$
	$\cH^s_\eps$ to $\eps^{1/2}H^s$, and $[0,\eps]$ to $[0,1]$,
whereupon the result follows by well-posedness of linear diffusion equation
$\D_t v   +  g^\eps v     -   \Delta_x v= f$, smallness of initial data $w_0$, and standard Picard iteration.
\end{proof}

\begin{proof}[Proof of Corollary \ref{main3}]
	First note, as in the proof of Proposition \ref{prepprop},
	that $\psi$ and $u^{\eps,m}$ may be extended without loss of generality to the interval
	$[-\eps, T+\eps]$.
	Letting $\chi(z)$ be a smooth cutoff function vanishing for $z\geq 1$ and $1$ for $z\leq 0$
	and $\chi^\eps(t):=\chi(t/\eps)$, set
	$$
	\tilde v:=  (1-\chi^\eps(t)) v + \chi^{\eps} w , 
	$$
	where $w$ is the solution described in Lemma \ref{fastlem}, 
	with $w_0= w|_{t=0}\equiv 0$.

Substituting in \eqref{veq} and rearranging, we obtain
	\be\label{tildeeq}
	\eps \cP_{u^{\eps,m}}\tilde v= \eps \tilde e^{\eps,m}:= \tilde R^{\eps,m} + 
	\tilde Q(u^{\eps,m}, \tilde v),
	\ee
where
$
\tilde R^{\eps,m} = (1-\chi^\eps(t))R^{\eps,m} + O\big(\chi^{\eps})'(t)w\big) + O\big((1-\chi^\eps(t))w\big)
$
and
$ \tilde Q(u^{\eps,m}, \tilde v):=\chi^\eps(t-T) Q(u^{\eps,m}, \tilde v)$.  

	By \eqref{wbd} with $w_0=0$, we have $\|\tilde R^{\eps,m}\|_{\cH^s}\lesssim \|R^{\eps,m}\|_{\cH^s}$.
	Moreover, $(1-\chi^\eps)$ and $(\chi^\eps)'$, hence also $\tilde R^{\eps,m}$, all vanish
	for $t<0$.
	Thus, by Theorem \ref{main2}, we may express \eqref{tildeeq} as 
	$
	\tilde v= (\cP_{u^{\eps,m}})^{-1} \tilde e^{\eps,m} ,
	$
	or, rescaling as in \eqref{barmeq}, the fixed-point equation
	$$
	\bar{\tilde v}=\cT \bar{\tilde v}:=  \cP_{u^{\eps,m}}^{-1} \big( -  \eps^{-(m+1)}\tilde R^{\eps,m}  
	+ \eps^{m-1} \bar {\tilde Q}(u^{\eps,m}, \bar {\tilde v}) \big),
	$$
	for $\bar{\tilde v}:= \eps^{-m} {\tilde v}$ vanishing for $t\leq 0$.
	Arguing as in the proof of Proposition \ref{prepprop}, we find that $\cT$ is contractive on 
	a ball of radius $\lesssim \|\eps^{-(m+1)}R^{\eps,m}\|_{\cH^s_\eps}$,  
	with Lipschitz constant $\lesssim \eps^{m-3/2}\lesssim \eps^{1/2}=o(1)$,
	hence there exists a unique solution $ v^{m,\eps}\in \cH^s_\eps$ vanishing for $t\leq 0$ and
	satisfying \eqref{rembd}.
\end{proof}

\br\label{attractrmk}
Evidently, by \eqref{wbd}, we could choose any initial data $w_0$ satisfying
$$
\|w_0\|_{\cH^s_\eps(\R^d)} \lesssim \|R^{\eps,m}\|_{\cH^s_\eps(\R^d\times [0,T])} \lesssim 
\eps^{m+1}\|\D_{t,x}^2 \psi \|_{H^s}
$$
without affecting the argument above, hence the conclusion 
$$
\|u-u^{\eps,m}\|_{\cH^s_\eps}\lesssim \eps^m \|\D_{t,x}^2 \psi\|_{H^s}
$$
of Corollary \ref{main3} holds for any $u$ satisfying 
$
\|u|_{t=0}-u^{\eps,m}|_{t=0} \|_{\cH^s_\eps(\R^d)}\lesssim \eps^{m+1}\|\D_{t,x}^2 \psi \|_{H^s}.
$
\er


\appendix

\section{Kreiss-type construction for $2\times 2$ blocks}\label{s:kcon}
Here, under Assumption \ref{ass3}, we construct a smooth symmetrizer $s(\lambda)$
	 in the sense of Proposition \ref{symm} 
	 for a $2\times 2$ block
$$
m(\lambda,k)=
	J + \begin{pmatrix} a(\kappa,\lambda)\lambda  & b(\kappa,\lambda)\lambda\\ 
	\lambda + c(\kappa,\lambda)\lambda^2 + e_0(\kappa) \kappa \lambda  & d(\kappa,\lambda)\lambda \end{pmatrix}
+ \kappa n(\kappa),
 \qquad
 J:= \begin{pmatrix} 0 & 1\\ 0 & 0 \end{pmatrix},
$$
where $a= a_1+a_2i$, $b=b_1+b_2i$, $c=c_1+c_2i$, $d=d_1+d_2i$, $e_0=e_{0,1}+e_{0,2}i$ 
are smooth functions of $\lambda=\gamma+i\tau$ and $\kappa$, $n$ is a smooth function of $\kappa$,
and $\gamma, \tau, \kappa$ are real and sufficiently small.
That is, assuming that the small spectral curve
\be\label{stabass}
\lambda_*(\xi)=\theta(\kappa)i\xi -\eta(\kappa)\xi^2+ \dots
\ee
determined by $d(\lambda_*(\xi),\lambda,\kappa)\equiv 0$ where $d(\lambda,\kappa):=
\det(e^{m(\lambda,\kappa)X}-e^{i\xi X})$, or equivalently 
\be\label{equivstab}
\det \big( m(\lambda_*(\xi), \kappa)- i\xi \big)\equiv 0,
\ee
satisfies the conditions
\be\label{mstab}
\hbox{\rm $\theta$ real, $\eta$ real and positive}
\ee
inherited from Assumption \ref{ass3} and the structure of the original problem,
we seek
\be\label{sdef}
s(\lambda,\kappa)=
\begin{pmatrix} \alpha & 1+ i\sigma\\ 1-i\sigma & \beta \end{pmatrix}(\lambda,\kappa)
\ee
	 smooth with respect to $(\lambda,\kappa)$, with $\alpha, \beta, \sigma$ real such that for $\kappa$, $\lambda$ small enough, $\Re (sm)\gtrsim \gamma + |\tau|^2$.

\begin{lemma}\label{nlem}
Assuming \ref{mstab}, there exists a smooth change of coordinates $T(\kappa)=\Id + \kappa T_1(\kappa)$
	such that $J+\kappa n(\kappa)$ is transformed to a $(1+O(\kappa))$ multiple of 
$\begin{pmatrix} 0 & 1 \\ 0 & f \kappa\end{pmatrix}$, where $f=f(\kappa)$ is real.
\end{lemma}

\begin{proof} Write $J+ n(\kappa)$ as 
$\begin{pmatrix} \tilde a \kappa & 1 + \tilde  b \kappa \\ \tilde c\kappa & \tilde d \kappa\end{pmatrix}$.
Rescaling by factor $(1+\tilde b\kappa)^{-1}$, we may take without loss of generality $\tilde b\equiv 0$.
Setting $T=\begin{pmatrix} 1 & 0\\\tilde a  \kappa & 1\end{pmatrix}$, and computing, we obtain
	$$
	\begin{aligned}
	T (J+\kappa n(\kappa))T^{-1}&=
\begin{pmatrix} 1 & 0\\\tilde a \kappa & 1\end{pmatrix}
\begin{pmatrix} \tilde a \kappa & 1 \\ \tilde c\kappa & \tilde d \kappa\end{pmatrix}
\begin{pmatrix} 1 & 0\\ -\tilde a \kappa & 1\end{pmatrix}
	= \begin{pmatrix} 0 & 1 \\
	(\tilde c- \tilde a \tilde d)\kappa^2 &(\tilde a + \tilde d) \kappa\end{pmatrix}
	\end{aligned}
		$$
	Recalling that $\det(J+\kappa n)=0$, by $\lambda_*(0)=0$, we have that $\tilde c=\tilde a \tilde d$, whence,
	defining $f:\tilde a + \tilde d$, we obtain the asserted form.
	Noting that $\cM$ for $\lambda=0$ is real-valued, we find that its eigenvalues are real or else
	occur in conjugate pairs.  As there are only two small eigenvalues, and one is zero, the other must be real as 
	well.  But $f\kappa$ by inspection is the other eigenvalue, whence $f$ is real.
\end{proof}

By Lemma \ref{nlem}, we are reduced to finding a smooth symmetrizer \eqref{sdef} for
$$
m_1(\lambda,k)= TmT^{-1}=
\begin{pmatrix}  a(\kappa,\lambda)\lambda  & 1+  b(\kappa,\lambda)\lambda\\ 
\lambda + c(\kappa,\lambda)\lambda^2 + e_0(\kappa) \kappa \lambda  & 
d(\kappa,\lambda)\lambda + f(\kappa)\kappa \end{pmatrix}
$$
with 
$f(\kappa)$ real.
Setting $\tilde T=\begin{pmatrix} 1 & 0\\ a\lambda/(1+b \lambda) & 1\end{pmatrix}$, and 
	making the coordinate transformation 
	$$\tilde m:= \tilde T m \tilde T^{-1}$$
	similarly as in the proof of
	Lemma \ref{nlem}, we may simplify further to form
\be\label{mtilde}
\tilde m(\lambda,k)=
\begin{pmatrix}  0  & 1+  b(\kappa,\lambda)\lambda\\ 
	\lambda + c(\kappa,\lambda)\lambda^2 + e_0(\kappa) \kappa \lambda 
& d(\kappa,\lambda)\lambda + f(\kappa)\kappa \end{pmatrix}.
\ee

We will require the following key observation relating $e_0$, $f$ to expansion \eqref{stabass}.

\begin{lemma}\label{aereal}
Assuming \ref{mstab}, $e_0$ is real if $\theta \kappa \neq 0$ and $\kappa f=0$ if $\theta \kappa =0$.
	In either case, $e_{0,2}f\kappa=0$.
\end{lemma}

\begin{proof}
Substituting \eqref{mtilde} into \eqref{equivstab} gives
$
	0= (-i\xi)(d\lambda_*-f\kappa)- \lambda_*(1+b\lambda_*)(1+c\lambda_* + e_0 \kappa).
$
Substituting \eqref{stabass} and solving to order $\xi$ for $\kappa$ fixed then gives
$  f \kappa - \theta \kappa (1+e_0\kappa)=0$, and thus
	for $\theta \kappa \neq 0$, $ e_0= (f-\theta)/\kappa \theta={\rm real} $
	by reality of $f$, $\theta$, $\kappa$, or $e_{0,2}=\Im e_0=0$. 
	Similarly, for $\theta \kappa =0$ we have $f\kappa=0$.
	Combining cases, we have $e_{0,2} f\kappa =  0$ in either scenario.
\end{proof}

We are now ready to prove existence of symmetrizers. Computing, we have for $|\lambda|, |\kappa|=o(1)$,
\ba\label{smcomp}
	\Re(sm)&=
\Re
\begin{pmatrix} 
	(1 + i\sigma)(\lambda + c\lambda^2+ e_0\kappa \lambda)
	& \alpha(1 + b\lambda) + (1+i\sigma)(d\lambda +f\kappa) \\
	\beta(\lambda +c \lambda^2 + e_0\kappa \lambda) & 
	(1 -i\sigma)(1+b \lambda)  + \beta d\lambda + \beta f\kappa
 \end{pmatrix}\\
	\quad & =
	\begin{pmatrix} Y 
		& X\\ \bar X &  1+o(1) + O(\sigma) + o(1)\beta
 \end{pmatrix},
 \ea
where, separating out $\gamma$ contributions, we have, for $\alpha, \beta, \sigma$ bounded, 
$$
\begin{aligned}
	X &= O(\gamma)+ \alpha(1+bi\tau)+ (1+i\sigma) (di \tau + f\kappa) 
	+ \beta(-i\tau -\bar c \tau^2 - \bar e_0 i\tau) \\
	Y&= 
	\gamma \Re \big( (1+i\sigma)(1+ c(\gamma + 2i\tau)+ e_0 \kappa)
	+ 
\tau \Re \big(  (1+i\sigma)(i-c\tau+ e_0\kappa i) \big)\\
	&= \gamma(1+O(|\tau,\kappa|(1+|\sigma|)) 
	+ \tau \big(  -c_1\tau - e_{0,2}\kappa  +\sigma(-1 +2c_2\tau - e_{0,1}\kappa ) 
	\big),\\
\end{aligned}
$$
giving, for $|\sigma|$ bounded and $\kappa, \tau$ small,
	\ba\label{XYvals}
	\Re X &= \alpha(1- b_2\tau) + \beta(e_{0,2}\kappa \tau- c_1 \tau^2)
	-\sigma d_1 \tau -d_2\tau +\kappa f + O(\gamma ) ,\\
	\Im X &= \alpha b_1 \tau  + \beta(- \tau + c_2\tau^2- e_{0,1}\kappa \tau) +
	\sigma(  - d_2\tau + f\kappa)
	+d_1\tau,\\
	Y&= \gamma(1+o(1)) + \sigma\tau(-1 + c_2\tau  - e_{0,1}\kappa)
	-e_{0,2}\kappa \tau -c_1 \tau^2. 
	\ea

Next, we seek $\alpha,\beta,\sigma$ such that non-$\gamma$ terms in \eqref{XYvals} are cancelled in
$\Re X$, $\Im X$, and give total contribution $|\tau|^2$ in $Y$, i.e.,
\ba\label{IFTsys}
\alpha &= \alpha b_2 \tau+ \beta( c_1 \tau^2 - e_{0,2}\kappa \tau) + \sigma d_1 \tau 
+d_2 \tau - \kappa f,\\
\beta &= \alpha b_1+ \beta( c_2\tau - e_{0,1}\kappa) + \sigma (f\kappa/\tau- d_2) + d_1, \\
\tau^2&= \sigma\tau(-1 + c_2\tau  - e_{0,1}\kappa) -e_{0,2}\kappa \tau -c_1 \tau^2,
\\
\ea
giving 
$
\Re (sm)= \begin{pmatrix} \gamma(1+o(1))+|\tau|^2 & O(\gamma)\\O(\gamma) &  1+o(1) 
 \end{pmatrix},
$
hence, by Sylvester's principal minor criterion, 
$$
\Re (sm)-(\gamma/2+|\tau|^2)\Id >0,
$$
and thus $\Re (sm)\gtrsim \gamma +|\tau|^2$ as desired, for $\gamma$, $\kappa $ sufficiently small.

It remains to solve the the linear system \eqref{IFTsys}.
Solving the third, decoupled equation as
\be\label{sigsol}
\sigma=(1-c_2\tau+e_{0,1}\kappa)^{-1}(-\tau-e_{0,2}\kappa -c_1 \tau) = O(|\tau, \kappa|),
\ee
and substituting into the second, we may rewrite the apparently singular term  $\sigma \kappa f/\tau$ as
$$
\begin{aligned}
	\sigma \kappa f/\tau& = 
(1-c_2\tau+e_{0,1}\kappa)^{-1}(-\tau -e_{0,2}\kappa -c_1 \tau)\kappa f/\tau\\
	&= -(1-c_2\tau+e_{0,1}\kappa)^{-1}(1+c_1) \kappa f= O(\kappa).
\end{aligned}
$$
Here, we have made crucial use of Lemma \ref{aereal} in observing the cancellation $e_{0,2}f\kappa=0$.
This reduces \eqref{IFTsys} to a linear system in $\alpha, \beta$ with bounded coefficients,
\ba\label{alphabeta}
\alpha &= \alpha b_2 \tau+ \beta( c_1 \tau^2 - e_{0,2}\kappa \tau) +d_2 \tau - \kappa f + O(|\tau, \kappa|)\tau,\\
\beta &= \alpha b_1+ \beta( c_2\tau - e_{0,1}\kappa) +  d_1 + O(\kappa), \\
\ea
or, for $\tau, \kappa$ sufficiently small,
\ba\label{matsys}
\begin{pmatrix} \alpha\\ \beta\end{pmatrix}&=
	\begin{pmatrix} 1+O(\tau)& O(|\tau,\kappa|)\\ -b_1 & 1 + O(|\tau,\kappa|) \end{pmatrix}^{-1}
		\begin{pmatrix} O(|\tau,\kappa|)\\ d_1 + O(|\tau,\kappa|)\end{pmatrix} \\
	&= \begin{pmatrix} 1& 0\\ b_1 & 1  \end{pmatrix}
		\begin{pmatrix} 0\\ d_1 \end{pmatrix} + O(|\tau,\kappa|) 
		= \begin{pmatrix} 0\\ d_1 \end{pmatrix} + O(|\tau,\kappa|), 
\ea
where $O(\tau)$, $O(\kappa)$, and $O(|\kappa,\tau|)$ terms are smooth functions of $(\kappa, \tau)$,
yielding
existence/uniqueness of smooth bounded solutions 
$ \alpha$, $\sigma = O(|\tau,k|)$ and $\beta =   d_1 + O(|\tau, \kappa|)$ of \eqref{IFTsys}.

This gives existence of smooth symmetrizers, completing the proof of Proposition~\ref{symm}.

\br\label{cormk}
A similar symmetrizer construction yielding the same bounds \eqref{symmeq} was carried out for $2\times 2$
``parabolic blocks'' in \cite[Lemma 5.2]{C} in the context of hyperbolic finite difference schemes.
An important new aspect complicating our analysis here is presence of the additional parameter $\kappa$.
\er

\section{Semiclassical paradifferential calculus}\label{s:paradiff_app}
For completeness, we include here essentially verbatim the useful r\'esum\'e of \cite[\S 4.2, p. 60-62]{GMWZ2},
describing the paradifferential calculi used here and in \cite{MZ,GMWZ1,GMWZ2}, etc.
For an extended version, including proofs, see \cite[\S3.1]{MZ}.  
These include both {\it homogeneous} and {\it parabolic} calculi, the former used for the bounded-frequency
and the latter for the large-frequency regime.
The reader may skip over the parts about parabolic calculus if desired, as this is used in the present paper only
implicitly, through reference to previous results of \cite{MZ}.

Each calculus has its own associated scaling.
With $\zeta=(\tau,\gamma,\eta)$ as before and
$\alpha=(\alpha_\tau,\alpha_\eta)\in\bN\times\bN^{d-1}$ a
multi-index, set $\bR^{d+1}_+=\{\zeta:\gamma\geq 0\}$ and
\begin{align}\label{j1}
\begin{split}
&\langle\zeta\rangle=(1+|\zeta|^2)^{1/2}\\
&\Lambda(\zeta)=(1+\tau^2+\gamma^2+|\eta|^4)^{1/4}\\
&|\alpha|=\alpha_\tau+|\alpha_\eta|,\;\;\|\alpha\|=2\alpha_\tau+|\alpha_\eta|.
\end{split}
\end{align}

\begin{defn}[Symbols]\label{j2}
\textbf{}

 1.   Let $\mu\in\bR$. The space of \emph{homogeneous}
symbols $\Gamma^\mu_0$ is the set of locally $L^\infty$ functions
$a(t,y,x,\zeta)$ on $\bR^{d+1}\times\bR^{d+1}_+$ which are
$C^\infty$ in $\zeta$ and satisfy:
\begin{align}\label{j3}
|\partial_{\tau,\eta}^\alpha a|\leq
C_\alpha\langle\zeta\rangle^{\mu-|\alpha|}, \text{ for all
}(t,y,x,\zeta) \text{ and }\alpha.
\end{align}

2.  For $k=0,1,2,\dots$, $\Gamma^\mu_k$ denotes the space of
symbols $a\in\Gamma^\mu_0$ such that $\partial^\alpha_{t,y}
a\in\Gamma^\mu_0$ for $|\alpha|\leq k$.

3.  The spaces of \emph{parabolic} symbols $P\Gamma^\mu_0$ and
$P\Gamma^\mu_k$ are defined in the same way, using
$\Lambda(\zeta)$ in place of $\langle\zeta\rangle$ and
$\|\alpha\|$ in place of $|\alpha|$.
\end{defn}

Observe that symbols in $\Gamma^\mu_k$ which
are independent of $x$ constitute a subspace of $\Gamma^\mu_k$,
and similarly for the spaces $P\Gamma^\mu_k$.

The spaces $\Gamma^\mu_k$ are equipped with the natural seminorms
\begin{align}\label{j4}
|a|_{\mu,k,N}:=\sup_{|\alpha|\leq N}\sup_{|\beta|\leq
k}\sup_{(t,y,x,\zeta)}\langle\zeta\rangle^{|\alpha|-\mu}|\partial_{t,y}^\beta\partial_\zeta^\alpha
a(t,y,x,\zeta)|.
\end{align}
Seminorms on the spaces $P\Gamma^\mu_k$ are defined in the
same way by the substitutions described earlier.

We consider a semiclassical quantization of symbols.  When
$a\in\Gamma^\mu_0$ is independent of $(t,y)$ the associated
homogeneous paradifferential operator acts in $(t,y)$ and is
defined by the Fourier multiplier $a(x,\epsilon\zeta)$:
\begin{align}\label{j5}
T^{\epsilon,\gamma}_a u(t,y,x)=\frac{1}{(2\pi)^d}\int
e^{it\tau+iy\eta} a(x,\epsilon\zeta)\hu(x,\tau,\eta)d\tau d\eta.
\end{align}
For $a\in P\Gamma^\mu_0$ the associated parabolic operator
$P^{\epsilon,\gamma}_a$ is defined by the same formula.  When the
symbols depend on $(t,y)$, the corresponding operators are defined
by formulas similar to \eqref{j5}, except that the symbols are
first smoothed in $(t,y)$ using an idea of Bony \cite{B}.  The
smoothing process in the homogeneous case differs from that in the
parabolic case (see \cite{MZ}, Proposition B.7). 
We shall often drop the superscripts $\epsilon,\gamma$ and write
the operator defined by \eqref{j5} as $T_a$.

\subsubsection{Sobolev spaces}  For $s\in\bR$ let $H^s$ denote the
space of functions $u(t,y)$ such that
\begin{align}\label{j7}
|u|_{s,\epsilon,\gamma}:=\left(\int_{\bR^d}\langle\epsilon\zeta\rangle^{2s}|\hu(\tau,\eta)|^2d\tau
d\eta\right)^{1/2}<\infty,
\end{align}
and let $\cH^s$ be the space of functions $u(t,y,x)$ such that
\begin{align}\label{j8}
\|u\|_{s,\epsilon,\gamma}=\left(\int
|u(\cdot,x)|^2_{s,\epsilon,\gamma} dx\right)^{1/2}<\infty.
\end{align}
Similarly define spaces $PH^s$ and $P\cH^s$ by substituting the
weight $\Lambda(\epsilon\zeta)$ for $\langle\epsilon\zeta\rangle$
in \eqref{j7} and \eqref{j8}. We use the same
notation $\|u\|_{s,\epsilon,\gamma}$ for norms in $\cH^s$ and
$P\cH^s$, depending on context for distinction.

\subsubsection{Action on Sobolev spaces, symbolic calculus}
\begin{prop}[Action]\label{j9}
For any $a\in\Gamma^\mu_0$ and $s\in\bR$ there is a $C$ such that
for $\epsilon\in (0,1]$, $\gamma\geq 1$ and $u\in\cH^s$:
\begin{align}\label{j10}
\|T_a^{\epsilon,\gamma}u\|_{s-\mu,\epsilon,\gamma}\leq
C\|u\|_{s,\epsilon,\gamma}.
\end{align}
The constant $C$ is bounded when $a$ remains in a bounded subset
of $\Gamma^\mu_0$.

For $a\in P\Gamma^\mu_0$ the operators $P^{\epsilon,\gamma}_a$
have the same mapping property on the spaces $P\cH^s$.
\end{prop}

\begin{prop}[Compositions]\label{j11}
Consider $a\in\Gamma^\mu_1$ and $b\in\Gamma^{\nu}_1$.  Then
$ab\in\Gamma^{\mu+\nu}_1$ and there is a $C$ such that for
$\epsilon\in (0,1]$, $\gamma\geq 1$ and $u\in\cH^s$:
\begin{align}\label{j12}
\|(T^{\epsilon,\gamma}_a\circ
T^{\epsilon,\gamma}_b-T^{\epsilon,\gamma}_{ab})u\|_{s-\mu-\nu+1,\epsilon,\gamma}\leq
C\epsilon\|u\|_{s,\epsilon,\gamma}.
\end{align}
The constant $C$ is bounded when $a$ and $b$ remain in bounded
subsets of $\Gamma^\mu_1$ and $\Gamma^\nu_1$ respectively.

Moreover, if $b$ is independent of $(t,y)$ then
$T^{\epsilon,\gamma}_a\circ
T^{\epsilon,\gamma}_b=T^{\epsilon,\gamma}_{ab}$.

The same inequality holds for compositions of operators
$P^{\epsilon,\gamma}_a$ and $P^{\epsilon,\gamma}_b$ acting on
$u\in P\cH^s$.
\end{prop}

\begin{prop}[Adjoints]\label{j13}
Let $a^*$ denote the adjoint of the matrix symbol
$a\in\Gamma^\mu_1$ and let $(T^{\epsilon,\gamma}_a)^*$ be the
adjoint operator of $T^{\epsilon,\gamma}_a$. There is a $C$ such
that for $\epsilon\in (0,1]$, $\gamma\geq 1$ and $u\in\cH^s$:
\begin{align}\label{j14}
\|((T^{\epsilon,\gamma}_a)^*-T^{\epsilon,\gamma}_{a^*}u\|_{s-\mu+1,\epsilon,\gamma}\leq
C\epsilon\|u\|_{s,\epsilon,\gamma}.
\end{align}

The same inequality is true for adjoints of operators
$P^{\epsilon,\gamma}_a$ acting on $u\in P\cH^s$.
\end{prop}

\begin{prop}[Commutators]\label{j14a}
For $a\in\Gamma^\mu_1$ and $u\in\cH^s$ we have
\begin{align}
[\partial,T^{\epsilon,\gamma}_a]u=T^{\epsilon,\gamma}_{\partial a}
u,
\end{align}
for $\partial=\partial_t$ or $\partial_{y_j}$.   A similar result
holds in the parabolic calculus.
\end{prop}

\begin{prop}[G\aa rding inequalities]\label{j15}
Consider symbols $a\in\Gamma^\mu_1$ and $w\in\Gamma^0_1$.  Suppose
that there is $\chi\in\Gamma^0_1$ and $c>0$ such that $\chi\geq
0$, $\chi w=w$ and
\begin{align}\label{j16}
\chi^2(t,y,x,\zeta)\Re a(t,y,x,\zeta)\geq
c\chi^2(t,y,x,\zeta)\langle\zeta\rangle^\mu\text{ for all
}(t,y,x,\zeta).
\end{align}
Then there is $C$ such that for all $\epsilon\in (0,1]$,
$\gamma\geq 1$ and $u\in\cH^{\mu/2}$:
\begin{align}\label{j17}
\frac{c}{2}\|T^{\epsilon,\gamma}_w
u\|^2_{\frac{\mu}{2},\epsilon,\gamma}\leq
\Re(T^{\epsilon,\gamma}_aT^{\epsilon,\gamma}_w
u,T^{\epsilon,\gamma}_w
u)_{L^2}+C\epsilon^2\|u\|^2_{\frac{\mu}{2}-1,\epsilon,\gamma}.
\end{align}

The same inequality holds for operators $P^{\epsilon,\gamma}_a$
acting on $u\in P\cH^{\mu/2}$.

\end{prop}

\subsubsection{Paraproducts}  Paraproducts are paradifferential
operators associated to symbols independent of $\zeta$.  The
following two propositions are used to estimate the errors
introduced in the passage from differential operators to their
paradifferential counterparts.  They can also be used to estimate
errors caused by passage from one calculus to the other.

\begin{defn}\label{j18}
For $k\in\bN$, let $\cW^k$ denote the space of functions
$a(t,y,x)$ on $\bR^{d+1}$ such that $\partial_{t,y}^\beta a\in
L^\infty(\bR^{d+1})$ for $|\beta|\leq k$.  
\end{defn}

Observe that
\begin{align}\label{j19}
\cW^k\subset \Gamma^0_k\cap P\Gamma^0_k.
\end{align}

\begin{prop}[Homogeneous paraproducts]\label{j20}
For any $a\in\cW^1$ there is a constant $C$ such that for all
$\epsilon\in (0,1]$, $\gamma\geq 1$ and $u\in\cH^1$:
\begin{align}\label{j21a}
\begin{split}
&\|au-T^{\epsilon,\gamma}_a u\|_{1,\epsilon,\gamma}\leq
C\epsilon\|u\|_{0,\epsilon,\gamma},\\
&\gamma\|au-T^{\epsilon,\gamma}_a
u\|_{0,\epsilon,\gamma}+\|a\partial
u-T^{\epsilon,\gamma}_a\partial u\|_{0,\epsilon,\gamma}\leq
C\|u\|_{0,\epsilon,\gamma}, \text{ for }\partial=\partial_t\text{
or }\partial_{y_j}.
\end{split}
\end{align}
\end{prop}

\begin{prop}[Parabolic paraproducts]\label{j21}
For any $a\in\cW^1$ there is a constant $C$ such that for all
$\epsilon\in (0,1]$, $\gamma\geq 1$ and $u\in P\cH^1$:
\begin{align}\label{j21b}
\begin{split}
&\|au-P^{\epsilon,\gamma}_a u\|_{1,\epsilon,\gamma}\leq
C\epsilon\|u\|_{0,\epsilon,\gamma},\\
&\|a\partial_{y_j} u-P^{\epsilon,\gamma}_a\partial_{y_j}
u\|_{0,\epsilon,\gamma}\leq C\|u\|_{0,\epsilon,\gamma},\\
&\gamma\|au-P^{\epsilon,\gamma}_a
u\|_{0,\epsilon,\gamma}+\|a\partial_t
u-P^{\epsilon,\gamma}_a\partial_t
u\|_{0,\epsilon,\gamma}+\sum_{|\alpha|=2}\epsilon\|a\partial_y^\alpha
u-P_a^{\epsilon,\gamma}\partial_y^\alpha
u\|_{0,\epsilon,\gamma}\leq C\|u\|_{1,\epsilon,\gamma}.
\end{split}
\end{align}
\end{prop}

\begin{rem}\label{j22}
The difference between the above two propositions is due to the
fact that the symbol $i\tau+\gamma$ is of order two in the
parabolic calculus, but of order one in the homogeneous calculus.
\end{rem}


\end{document}